\documentclass[11pt]{amsart}

\usepackage{amsmath,amssymb,amsfonts}
\usepackage[left=30mm,right=30mm,top=30mm,bottom=30mm]{geometry}

\usepackage{mathrsfs}
\usepackage{hyperref}
\usepackage{cite}
\usepackage[foot]{amsaddr}
\usepackage[sc]{mathpazo}
\usepackage{graphics}
\usepackage{graphicx}
\usepackage{color}
\usepackage{empheq} 
\usepackage{cases}
\usepackage{enumitem}
\usepackage{comment}

\newcommand{\N}{\mathbb{N}}
\newcommand{\R}{\mathbb{R}}
\newcommand{\C}{\mathbb{C}}
\newcommand{\M}{\mathbb{M}}
\renewcommand{\L}{\mathsf{L}^2}
\renewcommand{\H}{\mathsf{H}}
\newcommand{\HS}{\mathcal{H}}

\renewcommand{\a}{\mathfrak{a}}
\newcommand{\A}{{\mathcal{A}}}
\newcommand{\Res}{{\mathcal{R}}}
\newcommand{\J}{J} 

\newcommand{\B}{\mathscr{B}}
\renewcommand{\S}{\mathscr{S}}

\newcommand{\dom}{\mathrm{dom}}
\newcommand{\supp}{\mathrm{supp}}
\newcommand{\dist}{\mathrm{dist}}

\newcommand{\la}{\langle}
\newcommand{\ra}{\rangle}

\newcommand{\eps}{\varepsilon}
\newcommand{\e}{_{\varepsilon}}
\newcommand{\ke}{_{k,\varepsilon}}

\newcommand{\al}{\alpha}

\newcommand{\ga}{\gamma}

\renewcommand{\d}{\,\mathrm{d}}
\newcommand{\ds}{\displaystyle}

\newcommand{\Id}{\mathrm{I}}

\newcommand{\cupl}{\bigcup\limits}
\newcommand{\suml}{\sum\limits}
\newcommand{\liml}{\lim\limits}

\newcommand{\capty}{\mathrm{cap}}

\newcommand{\wt}{\widetilde}
\newcommand{\wh}{\widehat}

\newcommand{\ceq}{\coloneqq}

\newcommand{\ess}{{\rm ess}}
\newcommand{\disc}{{\rm disc}}

\newcommand{\restr}{\hspace{-3pt}\restriction}

\newcommand{\ssubset}{\subset\joinrel\subset}

\theoremstyle{plain}
\newtheorem{theorem}{Theorem}[section]
\newtheorem*{theorem*}{Theorem}
\newtheorem{lemma}[theorem]{Lemma}
\newtheorem*{lemma*}{Lemma}
\newtheorem{proposition}[theorem]{Proposition}
\newtheorem{corollary}[theorem]{Corollary}
\theoremstyle{remark}
\newtheorem{remark}[theorem]{Remark}
\newtheorem*{remark*}{Remark} 

\newtheorem*{example*}{Example} 
\theoremstyle{definition}
\newtheorem{definition}[theorem]{Definition} 

\numberwithin{equation}{section}
\numberwithin{figure}{section}

\begin{document}
	
	\title[Spectrum of the Laplacian on a domain perturbed by tiny resonators]{Spectrum of the Laplacian on a domain perturbed by small resonators}
	
	\author[Giuseppe Cardone]{Giuseppe Cardone\,$^1$}
	\address{$^1$  Department of Mathematics and Applications ``Renato Caccioppoli'', University of Naples Federico II, Naples, Italy} 
	\email{giuseppe.cardone@unina.it}

	\author[Andrii Khrabustovskyi]{Andrii Khrabustovskyi\,$^{2,3}$}
	\address{$^2$ Department of Physics, Faculty of Science, University of
		Hradec Kr\'{a}lov\'{e}, Hradec Kr\'{a}lov\'{e}, Czech Republic} 
	\address{$^3$ Department of Theoretical Physics,
		Nuclear Physics Institute of the Czech Academy of Sciences, \v{R}e\v{z}, Czech Republic} 
	\email{andrii.khrabustovskyi@uhk.cz}

\begin{abstract}
It is widely known that the spectrum of the Dirichlet Laplacian is stable under 
small perturbations of a domain, while in the case of  the Neumann or mixed boundary 
conditions the spectrum may abruptly  change.
In this work we discuss an example of such a domain perturbation.  
Let $\Omega$ be a (not {necessarily} bounded) domain in $\mathbb{R}^n$. We perturb it to
$  \Omega_\varepsilon=\Omega\setminus \cup_{k=1}^m  S_{k,\varepsilon},$
where  $S_{k,\varepsilon}$ are  closed surfaces with  small suitably scaled holes  (``windows'') through 
which the bounded domains enclosed by these surfaces (``resonators'') are connected to the outer domain. When $\eps$ goes to zero, the resonators shrink to points.
We prove that in the limit $\varepsilon\to 0$ the spectrum of the Laplacian on $\Omega_\varepsilon$ with the Neumann boundary conditions on  $S_{k,\varepsilon}$ and the Dirichlet boundary conditions on the outer boundary 
converges to the union of the spectrum of the Dirichlet Laplacian on $\Omega$ and the numbers
$\gamma_k$, $k=1,\dots,m$, being equal $1/4$ times the limit of the ratio between the capacity of the $k$th window and the volume of the $k$th resonator.
We obtain an estimate on  the rate of this convergence with respect to the  Hausdorff-type  metrics.  
Also,  an application of this result is presented: 
we construct an unbounded  waveguide-like  domain  with inserted resonators such that 
the eigenvalues of 
the  Laplacian  on this domain lying  below the essential spectrum threshold do coincide 
with prescribed numbers.  
\end{abstract}


\maketitle

\section{Introduction\label{sec:1}}

In this section we   provide
the  background motivating us to examine the present problem; then
we  sketch the main results. At the end we   give an overview of existing
literature concerning PDEs on domains with small resonators.

\subsection{Motivations}
In what follows, if $\Omega$ is a domain {(i.e., a connected and open set)} in $\R^n$, we denote by $-\Delta_\Omega^D$ and $-\Delta_\Omega^N$ the Dirichlet and the
Neumann Laplacians on $\Omega$, respectively. 
Also, if $\A$ is a self-adjoint operator with purely discrete spectrum bounded from below and accumulating at $\infty$, we denote by $\lambda_k(\A)$ its $k$th eigenvalue; as usual the eigenvalues are arranged in the ascending order and repeated according to their multiplicities.
\smallskip

It is known   that the spectrum of 
the Dirichlet Laplacian is stable under  
small perturbations of a domain (see, e.g., \cite[Theorem~1.5]{RT75} for a precise statement).
The situation becomes more subtle for the Laplacian with the Neumann or mixed boundary conditions.
The example below demonstrating this  goes back to the Courant-Hilbert
monograph \cite{CH53}; later it was elaborated in more details by Arrieta, Hale and Han in \cite{AHH91}.
Let $\Omega$ be a bounded domain in $\R^n$, $n\ge 2$. We perturb it to a domain $\Omega\e$ by 
attaching a small \emph{resonator} (in Remark~\ref{rem:reso} we will explain why such a name is reasonable) --
the union  of a bounded open set  $B\e$  of the diameter $\eps$  and 
a narrow passage $P\e$ of the length $\eps$ and  the cross-section diameter 
$\mathcal{O}(\eps^\al)$; see Figure~\ref{fig-CH}.
Since $\Omega\e$ and $\Omega$ differs only in a $\mathcal{O}(\eps)$-neighbourhood
of some point on $\partial\Omega$, the first naive guess is that the $k$th eigenvalues
of the Laplacians on $\Omega\e$ and $\Omega$ are close as $\eps\ll 1$.
This is indeed true for the Dirichlet case: 
it is not hard to show\footnote{One can prove 
\eqref{lambdaD:conv} using  \cite[Theorem~{1.5}]{RT75} and the fact that 
$\Omega\e$ converges metrically to $\Omega$.}
that 
\begin{gather}\label{lambdaD:conv}
\forall k\in\N:\quad\lambda_k(-\Delta_{\Omega\e}^D)\to \lambda_k(-\Delta_{\Omega}^D)\text{ as }\eps\to 0,
\end{gather}
while, for the Neumann eigenvalues this property fails:
one has $\lambda_1(-\Delta_{\Omega\e}^N)=0$, and for the next eigenvalues 
the following   result holds if $\al>\frac{n+1}{n-1}$  \cite[Theorems~4.1,4.2]{AHH91}:
$$\forall k\in\N\setminus\{1\}:\quad \lambda_k(-\Delta_{\Omega\e}^N)\to\lambda_{k-1}(-\Delta_{\Omega}^N)\text{ as }\eps\to 0,$$
in particular, $\lambda_2(-\Delta_{\Omega\e}^N)\to 0$ as $\eps\to 0$.

\begin{figure}[t]
\begin{center}
\begin{picture}(200,105)
\includegraphics[height=35mm]{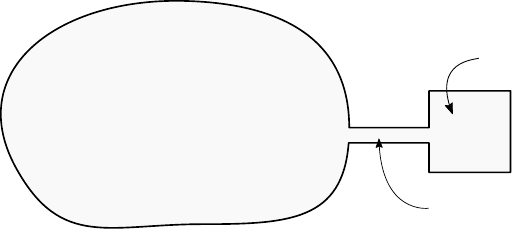}
\put(-150,60){$\Omega$}
\put(-35,5){$P\e$}
\put(-13,71){$B\e$}
\end{picture}
\end{center}
\caption{Attached resonator}\label{fig-CH}
\end{figure}  

\begin{figure}[h]
\begin{center}
\begin{picture}(150,100)
\includegraphics[height=35mm]{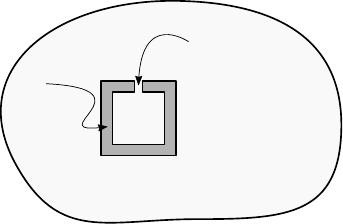}
\put(-25,30){$\Omega\e^{\rm out}$}
\put(-68,78){$P\e$}
\put(-143,60){$S\e$}
\put(-95,45){$B\e$}

\end{picture}
\end{center}
\caption{Inserted resonator}\label{fig-Sch}
\end{figure}
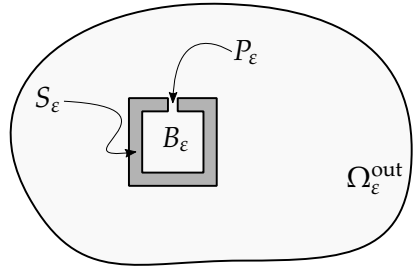

Instead of attaching a resonator, one can also insert it \emph{inside} $\Omega$.
Such a problem was studied by Schweizer in \cite{Sch15}. 
He considered   $\Omega\e=\Omega\setminus S\e$,
where $\Omega$ is a bounded domain in $\R^n$, and   $S\e$ is a thin layer separating 
the small set $B\e\subset\Omega$ and the outer domain $\Omega\e^{\rm out}$;  
$B\e$ and $\Omega\e^{\rm out}$  are connected through a narrow passage $P\e$. This domain is depicted on
Figure~\ref{fig-Sch}.
Denote by $L\e$ and $A\e$
 the length and the cross-section area of the passage $P\e$, respectively,
 and by $V\e$ the volume of $B\e$.
Let $\A\e$ be the  Laplacian on $\Omega\e$ subject to the Neumann conditions
on $\partial S\e$ and the Dirichlet conditions on $\partial\Omega$.
It was shown in \cite{Sch15} that, if $B\e$ and $P\e$ are appropriately scaled,  
one has
\begin{gather*}
\lambda_1(\A\e) \to\gamma ,\quad
\lambda_k(\A\e)\to\lambda_{k-1}(-\Delta_{\Omega}^D),\ k\in\N\setminus\{1\}\
\text{ as }\eps\to 0,
\end{gather*}
{where
\begin{gather}\label{gamma:Sch}
	 \gamma\ceq\lim_{\eps\to 0}{ \frac{A\e/L\e}{ V\e}};
\end{gather}}
the result is obtained under the assumption
that $\Omega$ is small enough in order to get
\begin{gather}
\label{ga:restr}
\ga<\lambda_1(-\Delta_\Omega^D).
\end{gather}
The proof in \cite{Sch15} relies  on variational methods.
An important ingredient is a  well-known result
for  compact  self-adjoint  operators  
claiming that the  existence  of  approximate  eigenfunctions  implies  the existence  of  nearby  eigenvalues;
in  \cite{Sch15} this result is applied for the resolvents of $\A\e$ and $-\Delta_{\Omega}^D$.
\smallskip

\emph{Our goal is to extend and complement the results obtained in 
\cite{Sch15} to resonators separated from the outer domain by a closed surface with a tiny windows.}
In the two next subsections we sketch the main results of this work.

\begin{figure}[h]
\begin{center}
\includegraphics[width=0.28\textwidth]{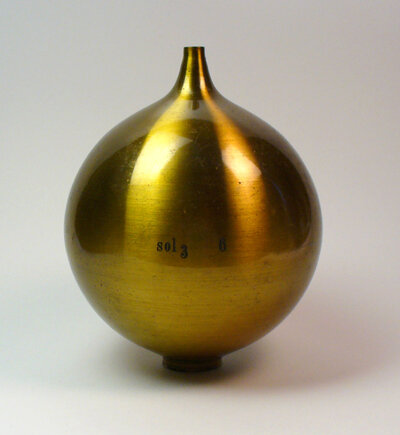}
\end{center}
\caption{Brass spherical Helmholtz resonator from around 1890-1900. 
 \href{https://commons.wikimedia.org/wiki/File:Helmholtz_resonator.jpg}{Photograph} by \href{https://en.wikipedia.org/wiki/User:Brian0918?rdfrom=commons:User:Brian0918}{brian0918}. 
 License: \href{https://creativecommons.org/licenses/by-sa/2.5/deed.en}{CC BY-SA 2.5}.}
\label{fig-Helm}
\end{figure}

\begin{remark}\label{rem:reso}
The reason for using the name ``resonator'' for these domain 
perturbations is that their geometry resembles the well-known
\emph{Helmholtz  resonators} --  acoustic  devices  
consisting of a resonator volume  being  connected by a thin channel to the outer space (Figure~\ref{fig-Helm}). The Helmholtz resonator is characterized by a  resonance frequency $\omega_{\rm res}$ which can be calculated by 
a widely known asymptotic formula
$$\omega_{\rm res}  \approx C\sqrt\frac{A}{LV}.$$
Here $C$ is the speed of sound, $V$ is the  resonator  volume,  
$L$ is the channel length and $A$ is the  channel  cross  section area.
Apparently, Schweizer's work \cite{Sch15} provides the first mathematically rigorous derivation of the resonator frequency formula. 
\end{remark}

\subsection{Sketch of the main convergence results}

Let $\Omega$ be a (not {necessarily} bounded) domain in $\R^n$, $n\ge 2$.  Let $m\in\N$. 
Let $\eps$ and $d\ke$, $k=1,\dots,m$ be small positive parameters such that $d\ke=o(\eps)$ as $\eps\to 0$. 
We introduce the sets
\begin{gather*}
S\ke=\partial B\ke\setminus D\ke.
\end{gather*}
Here $B\ke$ (``resonator'') is a subset of $\Omega$ such that $B\ke\cong\eps B_k$, where 
$B_k\subset\R^n$ is a bounded domain, 
$D\ke$ (``window'') is a subset of $\partial B\ke$ such that 
$D\ke\cong d\ke D_k$, where $D_k$ is a bounded set on an $(n-1)$-dimensional hyperplane. One has few other technical assumptions on the geometry of the above sets, whose description we postpone to Section~\ref{sec:2}. Removing the sets $S\ke$ from $\Omega$ we get the domain (Figure~\ref{fig-main})
$$\ds\Omega\e=\Omega\setminus\left(\cupl_{k=1}^m S\ke\right).$$ 
\begin{figure}[h]
\begin{center}
\begin{picture}(180,125)
\includegraphics{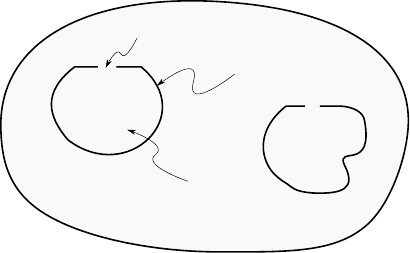}
\put(-83,83){${S\ke}$}
\put(-105,30){${B\ke}$}
\put(-135,106){${D\ke}$}
\end{picture}
\end{center}
\caption{The domain $\Omega\e$; here $m=2$}\label{fig-main}
\end{figure}

We denote by $\A\e$ the  Laplacian on $\Omega\e$ subject to the 
Dirichlet conditions on $\partial\Omega$ and the Neumann conditions on $S\ke$.
Assume that for each $k\in\{1,\dots,m\}$ the limit 
\begin{gather}\label{gamma:0}
\gamma_k\ceq \lim_{\eps\to 0}{\frac{\frac14\capty(D\ke)}{|B\ke| }}
\end{gather}
exists and is finite; here $\capty(D\ke)$ and $|B\ke|$ stand  for the capacity of $D\ke$ and 
the volume of $B\ke$, respectively. Further, in the Hilbert space $\L(\Omega)\oplus\C^m$ we introduce the operator 
$\A$ via 
$$
\A=(-\Delta_{\Omega}^D)\oplus\Gamma,
$$
where
$
\Gamma:\C^m\to\C^m$ is a diagonal matrix with the numbers $\gamma_k,\ k\in\M$, 
standing on the main diagonal.
Evidently,  one has
$$\sigma(\A)=\sigma(-\Delta_\Omega^D)\cup\{\gamma_k,\ k=1,\dots,m\}.$$

Our main result reads as follows.

\begin{theorem*}
The spectrum of   $\A\e$ converges to the spectrum of   $\A$ 
in the Hausdorff sense as $\eps\to 0$, i.e.

\begin{itemize}
\item
$\forall \lambda\in \sigma(\A)$ there exists a family $(\lambda\e)_{\eps>0}$ with 
$\lambda\e\in \sigma(\A\e)$ such that $\lambda\e\to \lambda$ as $\eps\to 0$, and

\item
$\forall \lambda\in\R\setminus \sigma(\A)$ there exists $\delta>0$ 
such that $\sigma(\A\e)\cap (\lambda-\delta,\lambda+\delta)=\varnothing$ for small enough $\eps$.
\end{itemize} 
\end{theorem*}

We derive an estimate on the rate of this convergence with respect to the 
Hausdorff metrics (see~\eqref{th1:est}). 
If the domain $\Omega$ is bounded, we get
\begin{gather*}
\forall k\in\N:\ \lambda_k(\A\e)\to\lambda_k(\A)\text{ as }\eps\to 0.
\end{gather*}
We also demonstrate the convergence of eigenfunctions.
The precise results are formulated in Section~\ref{sec:2} -- see Theorem~\ref{th1}, \ref{th2} and Corollary~\ref{coro:bounded}.

Note that, unlike \cite{Sch15} (cf.~\eqref{ga:restr}), 
we put no restrictions on the location of the numbers $\gamma_k$.
Also note that the type of the boundary conditions on the external boundary of $\Omega\e$ (i.e., on $\partial\Omega$) is not essential in our analysis, see Remark~\ref{rem:Neumann} at the end of Section~\ref{sec:2} for more details.

\subsection{Application: a waveguide with prescribed eigenvalues}
It is an interesting  and long-standing  topic of spectral theory -- 
the design of domains   with
prescribed spectral properties.
For example, in the celebrated paper \cite{CdV87} Y.~Colin de Verdi\'ere 
constructed a bounded domain $\Omega$ such that the 
first  $m$ non-zero eigenvalues of the Neumann Laplacian on $\Omega$ coincide with $m$ predefined pairwise distinct positive numbers.
In \cite{HSS91} Hempel, Seco and Simon  
constructed a bounded   domain such that the essential spectrum of the 
Neumann Laplacian on $\Omega$ coincides with the predefined closed set $S$;
in the case 
$0\in S$ the domain $\Omega$ looks like a chain of ``rooms''  and ``passages''.
This result was further elaborated by Hempel, Kriecherbauer and Plankensteiner in \cite{HKP97}, where also 
a prescribed bounded part of the discrete spectrum was realized; the 
designed domain has a ``comb'' structure.
We refer to the recent overview \cite{BK19}
for more references on the construction of Laplacians (and other self-adjoint operators) with the predefined 
(or partly predefined) spectrum.

We are  aimed to contribute to the above topic taking advantage of the convergence results we sketched in the previous subsection. These results suggest that by inserting small resonators one can create new eigenvalues having nothing in common with the eigenvalues of the Laplacian on the unperturbed domain. Moreover, it is easy to see that by a proper choice of the parameters $d\ke$ characterizing the windows sizes, one can make these new eigenvalue convergent ($\eps\to 0$) to predefined numbers.
It turns out that one can achieve even more -- \emph{the precise 
coincidence of these new eigenvalues with prescribed numbers} (for  fixed small enough $\eps$). To fix the ideas, we discuss an example demonstrating how to do this.

\begin{figure}[h]
\begin{center}
\begin{picture}(280,50)
\includegraphics[width=0.7\linewidth]{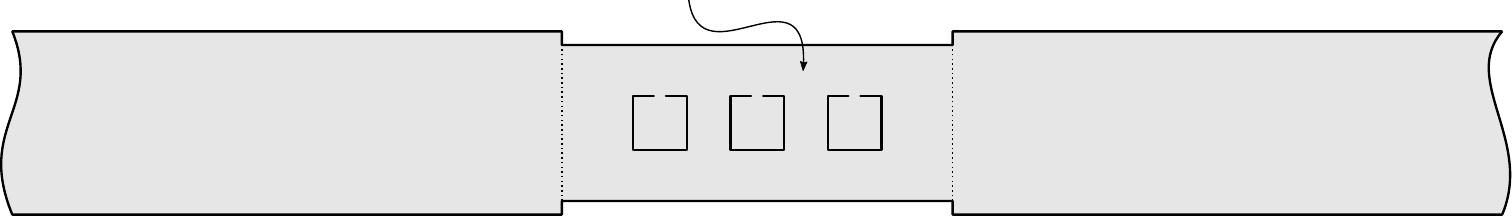}
\put(-55,15){$\Omega^+$}
\put(-220,15){$\Omega^-$}
\put(-146,40){$\wt\Omega\e$}
\end{picture}
\end{center}
\caption{Waveguide $\Omega\e$}
\label{fig-waveguide}
\end{figure}  

Let the unperturbed  domain be the straight unbounded tube (waveguide).   
It is known that the spectrum of the  Dirichlet Laplacian on such a domain 
coincides with $[\Lambda',\infty)$, where $\Lambda'>0$ is the smallest eigenvalue of the   Dirichlet Laplacian on the tube cross-section. We perturb $\Omega$ by narrowing the tube in some bounded part  
 and then by inserting $m$ resonators within this narrowed part (see Figure~\ref{fig-waveguide} -- here $\wt\Omega\e$ is a narrowed part with resonators, and $\Omega^\pm$ are semi-infinite straight tubes).
Such a perturbation does not change the essential spectrum, but may produce 
discrete eigenvalues below $\Lambda'$. \emph{We prove that the resonators actually can be chosen in such way that  these eigenvalues do coincide with prescribed numbers.} The role of the narrowing is to guaranteed that only $m$ eigenvalues appear below $\Lambda'$, and no further eigenvalues emerge in the vicinity of $\Lambda'$. 

The precise result is formulated in  Section~\ref{sec:5} -- see Theorem~\ref{th:exact}.
The proof is based on the multi-dimensional version
of the intermediate value theorem established in \cite{HKP97}. 
One of the key-points of the proof is the monotonicity of
eigenvalues of $\A\e$  with respect to the parameters $d\ke$ characterizing 
the sizes of the windows $D\ke$.
This is an advantage of the resonators we treat in this work 
comparing to the resonators as on Figures~\ref{fig-CH} and \ref{fig-Sch}, for which the monotonicity 
of eigenvalues (with respect to any of the involved geometrical parameters)  is not at all obvious.

\subsection{Methods}
As it was already mentioned, we do not assume the boundedness of $\Omega$,
whence the resolvents of $\A\e$ and $\A$ are non-compact operators in general.
Therefore, we cannot rely on the methods used in \cite{Sch15}.
Instead, our proofs utilize the abstract results for studying the convergence of operators
in varying Hilbert spaces developed in \cite{P06,P12,KP21}.
{In Section~\ref{sec:3} we formulate  Theorem~\ref{thAAA} which 
	is essentially the combination of these abstract results;  for the reader's convenience we recall these results in Appendix~\ref{appendix:A}
	and demonstrate how they imply Theorem~\ref{thAAA}
	in Appendix~\ref{appendix:B}}. 
The proof of the main results is given is Section~\ref{sec:4}.
Our main task is to show that {$\A\e$ and $\A$ satisfy the assumptions of  the aforementioned abstract Theorem~\ref{thAAA}}, which requires to construct suitable identification operators between
the Hilbert spaces $\L(\Omega\e)$ and $\L(\Omega)\oplus\C^m$, as well as between the energetic spaces associated with the quadratic form generating operators $\A\e$ and $\A$.

\subsection{Further literature on domains with resonators}

At the end of this section we give a brief overview of existing
literature concerning differential equations on domains 
perturbed via small resonators. Note that in some other works the authors, instead of  ``resonators'', 
use the names ``traps'' or ``accumulators''.
  
The interest in PDEs in domains with small resonators
arose in the early 90th  in connection with homogenization 
theory. It was observed that homogenizing Neumann problems in a domain perforated by a lot of
small resonators one  arrives on peculiar effective problems on an unperturbed domain. 
The first result was obtained by Khruslov in \cite{Khru89}, where
the linear heat equation was considered. The
effective equation in \cite{Khru89} contains a non-local in time term (model with memory). 
In the subsequent papers \cite{BCP99,BP97,Pa92,Ko95} non-linear PDEs were examined.
For more details we refer to the monograph \cite{MK06}.

One can also address another kind of homogenization problems, 
when resonators are not distributed inside a domain, 
but are attached along (a part of) its boundary.
We considered such a problem in \cite{CK15}, where homogenization 
of the Neumann spectral problem for the (weighted) Laplacian was studied.
The obtained effective problem contains spectral parameter 
both in the equation and the boundary conditions.

The remarkable applications of models with resonators is that they can be used to design materials with astonishing non-standard properties -- the so-called \emph{meta-materials}. Meta-materials consist 
of small components, and, 
even though the single component demonstrates ``standard'' behaviour, the meta-material behaves effectively in a way that is not known from ordinary materials.
Lamacz and Schweizer \cite{LS17}  studied 
acoustic properties of a medium partly filled by small resonators as on Figure~\ref{fig-Sch};
the mathematical model is the Helmholtz equation $-\Delta u\e=\omega^2 u\e$ subject to the Neumann conditions on the 
 resonators boundary (sound hard walls).
The obtained effective equation reads $-\nabla\cdot(A^*\nabla u)=\omega^2\Lambda(\omega) u$  with $A^*\gg 0$ and 
a frequency dependent coefficient $\Lambda(\omega)$ (dispersive medium). This coefficient can have negative values for some ranges of the frequency $\omega$, and it also  blows up when $\omega$ approaches the  resonance frequency of a single resonator (cf.~\cite{Sch15}).

Of course, Helmholtz-type resonators are not the only possible components for the design of meta-materials. One can also use, for example,  
wire structures (see, e.g., \cite{FB97}), 
split rings structures  (see, e.g., \cite{BS10}) or their mix. 
We refer to the overview \cite{Sch17} for more details and further references on this topic.

Finally, resonators constitute a good tool for opening of gaps in the spectrum of periodic differential operators.
Recall that the spectrum  of periodic self-adjoint differential operators
has the form of a locally finite union of compact intervals (\emph{bands}). 
An open interval on $\R$ is called a \emph{gap} 
if it has an empty intersection with the spectrum, but its endpoints belong to it.
The band structure of the spectrum suggests that gaps may exist in principle, but, in general, the presence of gaps is not guaranteed: two spectral bands may overlap, and then the corresponding gap disappears.
Existence of spectral gaps are of primary interest because of various applications, for example in physics of photonic crystals -- see, e.g., \cite{DLPSW11} for more details.
It was proven in \cite{Kh14,KK15} that the Neumann Laplacian 
on $\R^n$ with $m$ families of periodically distributed resonators has at least $m$ spectral gaps;
their location and lengths of can be  controlled by 
a suitable choice of the resonators sizes. 
Close results for waveguide-like domains with periodically attached 
resonators were obtained in \cite{CK17}.

\section{Setting of the problem and main results\label{sec:2}}
 
Let $n\in\N\setminus\{1\}$. In the following,
$x'=(x^1,\dots,x^{n-1})$ and $x=(x',x^n)$  stand for the Cartesian coordinates in $\R^{n-1}$ and $\R^{n}$, respectively.
For $r>0$ and $z\in\R^n$ we denote
$$\B(r,z)\ceq \left\{x\in\R^n:\ |x-z|<r\right\}.$$

Let $\Omega$  be a (not {necessarily} bounded)  domain in $\R^n$, $n\ge 2$.
Let $m\in\N$. We set   $$\M\ceq\{1,\dots,m\}.$$ 
Let
$B_k$, ${k\in\M}$ be  bounded Lipschitz domains in  $\R^{n}$. 
We assume that there exist positive numbers $\rho_k$, ${k\in\M}$  such that   
\begin{gather}
\label{DB1}
\forall k\in\M:\quad  \B(\rho_k,0)\cap B_k =  \B(\rho_k,0)\cap
\left\{x=(x',x^n)\in\R^n:\ x^n<0\right\},    
\end{gather}
whence, in particular, the boundary of $B_k$ is flat in the $\rho_k$-neighbourhood of the origin.
Further, let $\wt D_k$, ${k\in\M}$ be  bounded Lipschitz domains in  $\R^{n-1}$, and
\begin{gather}\label{Dk}
D_k\ceq \left\{x=(x',x^n)\in\R^n:\ x' \in  \wt D_k,\ x^n=0\right\}.
\end{gather}
We assume that 
\begin{gather}
\label{DB2}
\begin{array}{l}
\text{the smallest ball containing $D_k$ has center at the origin,}\\ 
\text{and its radius $\ell_k$ satisfies } 
\ell_k< \rho_k.
\end{array}
\end{gather}
In particular, it follows from \eqref{DB1}, \eqref{DB2} that ${D_k}\subset\partial B_k$.

Let $\eps>0$ be a small parameter. Let  $d\ke$, ${k\in\M}$ be positive numbers satisfying 
\begin{gather}\label{deps}
d\ke<\eps
\end{gather}
Below we specify $d\ke$ more precisely, see~\eqref{gamma:2}. Let $z_k$, ${k\in\M}$ be pairwise distinct points in $\Omega$. For $k\in\M$ we define
\begin{gather}\label{DBe}
B\ke\ceq\eps B_k+z_k,\quad
D\ke\ceq d\ke D_k+z_k.
\end{gather}
Due to \eqref{DB1}, \eqref{DB2}, \eqref{deps} 
we have  ${D\ke}\subset \partial B\ke$, and 
one has (see also Figure~\ref{fig-circles} in Section~\ref{sec:4}):
\begin{gather}\label{DBeps}
 \B(\rho_k\eps,z_k)\cap B\ke =  \B(\rho_k\eps,z_k)\cap
\left\{x=(x',x^n)\in\R^n:\ x^n<z_k^n\right\}, 
\end{gather}
where $z_k^n$ stands for the $x^n$-\,coordinate of $z_k$.
We assume that $\eps$ is sufficiently small so that 
\begin{gather*}
\overline{B\ke}\subset\Omega \quad\text{and}\quad
\overline{B\ke}\cap \overline{B_{l,\eps}}=\emptyset\text{ if }k\not=l.
\end{gather*}
Finally, we define the domain $\Omega\e$ (see Figure~\ref{fig-main}) by
\begin{gather}\label{Omega:e}
\Omega\e\ceq \Omega\setminus\left(\cupl_{k\in\M} S\ke\right),\text{ where }S\ke\ceq \partial B\ke\setminus D\ke.
\end{gather}

Now, we impose extra conditions on the sizes of the windows $D\ke$.
We assume that for each $k\in\M$ the  following limit 
exists and is finite:
\begin{gather}\label{gamma:1}
\gamma_k\ceq\liml_{\eps\to 0}\gamma\ke,\quad
\text{where }\gamma\ke\ceq \frac{\capty( {D\ke})}{4|B\ke|}.
\end{gather}
Here $|B\ke|$ stands for the volume of the domain $B\ke$, and 
$\capty({D\ke})$ stands for the capacity of the set $D\ke$;
the latter one is defined via
\begin{gather}\label{capty}
\capty(D\ke)=
\begin{cases}
\|\nabla H\ke\|_{\L(\R^n)}^2,&n\ge 3,\\
\|\nabla H\ke\|_{\L(\B(1,z_k))}^2,&n=2,
\end{cases}
\end{gather}
where $H\ke(x)$ is the solution to the problem
\begin{gather}\label{BVP:n3}
\begin{cases}
\Delta H(x)=0,&x\in\R^n\setminus\overline{D\ke},\\
H=1,&x\in\partial D\ke =\overline{D\ke},\\
H\to 0,&|x|\to\infty,
\end{cases}    
\end{gather}
if $n\ge 3$, or to the problem
\begin{gather}\label{BVP:n2}
\begin{cases}
\Delta H(x)=0,&x\in \B(1,z_k)\setminus\overline{D\ke},\\
H=1,&x\in\partial D\ke =\overline{D\ke},\\
H= 0,&x\in\partial \B( 1,z_k),   
\end{cases}    
\end{gather}
if $n=2$; here we assume that $\eps$ is sufficiently small in order to have  
$\overline{D\ke}\subset \B(1,z_k)$. Note that the capacity of $D\ke$ is positive despite its Lebesgue measure is zero.

One has the following asymptotics in the two-dimensional case 
\cite[Lemma~3.3]{CDG02}:
\begin{gather}\label{cap:rescaling1}
\capty( {D\ke})=2\pi|\ln d\ke|^{-1}(1+\al(\eps))\text{ with } \lim_{\eps\to 0}\al(\eps)= 0 ,\ n=2 ,
\end{gather}
while in higher dimensions, by using simple re-scaling arguments, we get
\begin{gather}\label{cap:rescaling2}
\capty( {D\ke})=(d\ke)^{n-2}\capty( {D_k}),\ n\geq 3.
\end{gather}
It follows from \eqref{gamma:1}, \eqref{cap:rescaling1}, \eqref{cap:rescaling2} that
\begin{gather}\label{gamma:2}
d\ke= C\e\eps^{\frac{n}{n-2}}\text{ if }n\ge 3
\qquad\text{and}\qquad
|\ln d\ke|^{-1}= C\e\eps^2\text{ if }n= 2,
\end{gather}
where $C\e=\mathcal{O}(1)$ as $\eps\to 0$. 
 
To define the  Laplace operator on $\Omega\e$ subject to the Neumann boundary conditions on $\cup_{k\in\M}S\ke$ and the Dirichlet  boundary  conditions on  $\partial\Omega$, we need an appropriate space, which we denote $\H_{0,\partial \Omega}^1(\Omega\e)$. {
We define it as the closure in $\H^1$ norm of the set consisting of functions from $\H^1(\Omega\e)$ 
vanishing in a neighborhood of $\partial\Omega$.}
If $\partial\Omega$ is sufficiently regular, one has
$$\H_{0,\partial \Omega}^1(\Omega\e)=\left\{u\in \H^1(\Omega\e):\  u\restr_{\partial\Omega}=0\right\},$$ 
where $u\restr_{\partial\Omega}$ is understood in the sense of traces.

Now, we can introduce the operator $\A\e$.
In  the Hilbert space $\L(\Omega\e)$ we define the sesquilinear form $\a\e$ by
\begin{equation}\label{ae}
\a\e[u,v]=\int_{\Omega\e}\nabla u\cdot\overline{\nabla v}\,\d x,\qquad 
\dom(\a\e)=\H_{0,\partial \Omega}^1(\Omega\e).
\end{equation}
The form $\a\e$ is densely defined in $\L(\Omega)$, nonnegative, and closed.
By the first representation
theorem (see, e.g. \cite[Chapter 6, Theorem 2.1]{Ka66}) there is a unique nonnegative
self-adjoint operator $\A\e$ in $\L(\Omega)$ associated with $\a\e$, i.e.,  
$\dom(\A\e)\subset\dom(\a\e)$ and  
\begin{equation*}
(\A\e u, v)_{\L(\Omega\e)}=\a\e[u,v],\quad \forall u\in \dom(\A\e),\,\,v\in\dom(\a\e).
\end{equation*}

Further, we define the  limiting operator $\A$. 
In the following, for $f\in\L(\Omega)\oplus \C^m$ we denote its
$\L(\Omega)$  and  $\C^m$ components by $f_0$ and $(f_k)_{k\in\M}$, respectively.
We define the sesquilinear form $\a$ in $\L(\Omega)\oplus \C^m$ by
\begin{gather}\label{a}
\ds\a[f,g]=\int_{\Omega}\nabla f_0\cdot\overline{\nabla g_0}\,\d x +
\suml_{k\in\M}\gamma_k f_k\overline{g_k},
\quad\dom(\a)= \H^1_0(\Omega)\times \C^m.
\end{gather}
This form is   densely defined in $\L(\Omega)\oplus \C^m$, nonnegative, and closed.
We denote by $\A$ the  nonnegative 
self-adjoint operator   in $\L(\Omega)\oplus \C^m$ associated with $\a$.
It is easy to see that 
\begin{gather}
\label{A:oplus}
\A=(-\Delta^D_\Omega)\oplus \Gamma\text{\quad and\quad}\sigma(\A)=\sigma(-\Delta_\Omega^D)\cup \{\gamma_k,\, k\in\M\},
\end{gather}
where 
$\Gamma:\C^m\to\C^m$ acts on $f=(f_k)_{k=1}^m\in\C^m$ by
$$\ds(\Gamma f)_k= \gamma_k f_k,\ k\in\M.$$
Our aim is to show that  $\sigma(\A\e)$ converges to  $\sigma(\A)$
in a suitable sense as $\eps\to 0$.
\smallskip

Recall that
for closed sets $X,Y\subset\R$  the \emph{Hausdorff distance} $d_H (X,Y)$ is given by
\begin{gather*}
d_H (X,Y)\ceq\max\left\{\sup_{x\in X} \inf_{y\in Y}|x-y|;\,\sup_{y\in Y} \inf_{x\in X}|y-x|\right\}.
\end{gather*}
The notion of convergence provided by this metric   is too restrictive for our purposes, since
the convergence of $\sigma(\A\e)$ to  $\sigma(\A)$  in  the metric $d_H(\cdot,\cdot)$ would mean  that  $\sigma(\A\e)$ and $\sigma(\A)$   look nearly the same \emph{uniformly in the whole  of $[0,\infty)$} -- a situation which is usually not guaranteed even if $\A\e$ converges to $\A$ 
in (a kind of) the norm  resolvent  topology.
Therefore, it is convenient to introduce the new metric $\widetilde{d}_H(\cdot,\cdot)$ via
\begin{gather*}
\widetilde{d}_H(X,Y)\ceq d_H( \overline{(1+X)^{-1}}, \overline{(1+Y)^{-1}}),\ X,Y\subset[0,\infty),
\end{gather*}
where  $(1+X)^{-1}\ceq {\{(1+x)^{-1}:\ x\in X\}}$, $(1+Y)^{-1}\ceq {\{(1+y)^{-1}:\ y\in Y\}}$.
With respect to this metric  two  spectra  can  be  close  even  if  they  differ  significantly  at  high energies. 

To  clarify  the  statement $\widetilde{d}_H(X\e,X)\to 0$  we introduce the following definition.

\begin{definition}
Let $X\e,X\in\R$. One says that 

\begin{itemize}
\item \emph{$X\e$ converges from inside to $X$} (we denote $X\e\nearrow X$)
if for any $ x\in X$ there exists a family $(x\e)_{\eps>0}$ with 
$x\e\in X\e$ such that $x\e\to x$ as $\eps\to 0$;

\item \emph{$X\e$ converges from outside to $X$} (we denote $X\e\searrow X$)
if for any $x\in\R\setminus X$ there exist $\delta>0$ 
such that $X\e\cap (x-\delta,x+\delta)=\varnothing$ for sufficiently small $\eps$.

\item \emph{$X\e$ converges  to $X$} (we denote $X\e\to X$)
if simultaneously $X\e\nearrow X$ and $X\e\searrow X$.
\end{itemize}
\end{definition}

\begin{remark}\label{rem:equiv}
It is easy to see that $X\e\searrow X$ iff
for any sequence $(x_{\eps_k})_{k\in\N}$ with $x_{\eps_k}\in X_{\eps_k}$   converging to some $x\in\R$ as $\eps_k\to 0$, one has $x\in X$.
\end{remark}

\begin{lemma}[{\cite[Lemma~A.2]{HN99}}]  \label{lemma:HN}
Let $(X\e)_{\eps>0}$ be a family of closed sets, $X\e\subset[0,\infty)$.
Let $X\subset [0,\infty)$ be a closed set. Then
$
\widetilde{d}_H(X\e,X)\to 0$ if and only if $X\e\to X$.
\end{lemma}

We are now in position to formulate the main convergence results. 
In what follows, by $C,C_1,C_2,\dots$ we denote generic positive constants being independent of $\eps$.

\begin{theorem}\label{th1}
One has
$\sigma(\A\e)\to \sigma(\A)$ as $\eps\to 0$, and  the estimate
\begin{gather} \label{th1:est}
\widetilde{d}_H\left(\sigma(\A\e),\sigma(\A)\right)\leq 
\begin{cases}
\ds C \sum_{k\in\M}|\ga\ke-\ga_k|+C\eps ,&n\ge 3,\\
\ds C \sum_{k\in\M}|\ga\ke-\ga_k|+C\eps|\ln\eps|^{3/2} ,&n=2,
\end{cases}
\end{gather}
holds for sufficiently small $\eps$.
\end{theorem}

As usual, we denote by
$\sigma_{\ess}(\cdot)$ and $\sigma_{\disc}(\cdot)$  the  essential and the discrete spectra of a self-adjoint operator.
Using standard methods of perturbation theory, one infers that inserting 
finitely many compact Lipschitz surfaces $S\ke$ inside $\Omega$ and posing the Neumann boundary conditions on them, one does not change the essential spectrum of the Laplacian, i.e. 
\begin{gather}\label{essAA}
\sigma_\ess(\A\e)=\sigma_\ess(-\Delta_\Omega^D)=\sigma_\ess(\A)
\end{gather}
(the second equality in \eqref{essAA} follows from \eqref{A:oplus}). For  the discrete spectrum of
$\A\e$ we have the following result.

\begin{theorem}\label{th2}
One has 
$
\sigma_\disc(\A\e)\nearrow \sigma_\disc(\A)\text{ as }\eps\to 0.$
The multiplicity is preserved {in the following sense}:
if $\lambda\in\sigma_\disc(\A)$ is of multiplicity $\mu$ and
$[\lambda-L,\lambda+L]\cap\sigma(\A)=\{\lambda\}$ with $L>0$, then for sufficiently small $\eps$
  the spectrum of $\A\e$ in $[\lambda-L,\lambda+L]$ is purely discrete and 
 the total multiplicity of the eigenvalues of $\A\e$ contained in $[\lambda-L,\lambda+L]$ equals $\mu$. 
 
 If, in addition, $\mu=1$ (i.e.  the eigenvalue $\lambda$ is simple), and  $\psi=(\psi_0,\psi_1,\dots,\psi_m)$ with $\psi_0\in\L(\Omega)$ and $(\psi_k)_{k\in\M}\in \C^m$ is the corresponding normalized in ${\L(\Omega)\oplus\C^m}$ eigenfunction,  
 then there exists a sequence of   normalized in $\L(\Omega\e)$ eigenfunctions $\psi\e$ of $\A\e$  such that 
\begin{gather*} 
\|\psi\e - \psi_0\|^2_{\L(\Omega\e\setminus \overline{\cup_{k\in\M}B\ke})}+
\suml_{k\in\M}\|\psi\e - |B\ke|^{-1/2}\psi_k\|^2_{\L(B_{k,\eps})} \to 0.
\end{gather*}
\end{theorem} 

\begin{remark}Let  $\lambda$ be a simple eigenvalue of $\A$, and let $\lambda\e\in\sigma_\disc(\A\e)$ with $\lambda\e\to\lambda$ as $\eps\to 0$. The second part of
Theorem~\ref{th2} asserts that the normalized eigenfunction, which corresponds to $\lambda\e$, concentrates on ${\L(\Omega\e\setminus \overline{\cup_{k\in\M}B\ke})}$
if $\lambda$ is an eigenvalue of $-\Delta_\Omega^D$ and concentrates on $B\ke$ if $\lambda=\gamma_k$ for some $k\in\M$.
\end{remark}

\begin{remark}  
If $\Omega$ is a  bounded domain, the spectra of $\A\e$ and $\A$ are purely discrete;
hence, due to Theorem~\ref{th1}, additionally to $
\sigma_\disc(\A\e)\nearrow \sigma_\disc(\A)$  we also have 
$\sigma_\disc(\A\e)\searrow\sigma_\disc(\A)$ as $\eps\to 0$.
However,  if  $\Omega$ is unbounded, 
the latter property does not  {necessarily} hold true: it may happen that
there exists a sequence $(\lambda_{\eps_k})_{k\in\N}$ with $\lambda_{\eps_k}\in\sigma_{\disc}(\A_{\eps_k})$ converging to  $\lambda_0\in\sigma_{\ess}(\A)$ as $\eps_k\to 0$. 
In Subsection~\ref{subsec4:4} we present an example demonstrating this. 
\end{remark}

From Theorems~\ref{th1} and \ref{th2} one can easily get the following corollary.

\begin{corollary}\label{coro:bounded}
Let $\Omega$ be a bounded domain (consequently, the spectra of $\A\e$ and $\A$ are purely discrete).
Then $\forall k\in\N$  one has $\lambda_k(\A\e)\to\lambda_k(\A)$ as $\eps\to 0$.
\end{corollary}

\begin{remark}\label{rem:Neumann}
The choice of the boundary conditions on the external boundary of $\Omega\e$ (i.e., on $\partial\Omega$)  plays no essential role in our analysis.
If one prescribes the Neumann, Robin  or mixed boundary conditions, all convergence results (Theorems~\ref{th1}, \ref{th2} and Corollary~\ref{coro:bounded}) remains the same, but 
with $-\Delta_\Omega^D$ being replaced by the Laplacian on $\Omega$ subject to the Neumann, Robin or mixed boundary conditions, respectively. Up to a few simple technical details the proofs are similar to the Dirichlet case, and to simplify the presentation we omit them.
\end{remark}

{ 
\begin{remark} 
It is interesting to compare our formula \eqref{gamma:0} for the resonance eigenvalues with the one obtained in \cite{Sch15},  see \eqref{gamma:Sch}. In both formulae one has the resonator volume in the denominator, while the numerators look very different at first glance. However, it turns out that both numerators can be written in a uniform way in terms of a certain quantity characterizing conductivity of the set through which the resonator is connected to the outer space.
Below we briefly explain the idea.

For the sake of simplicity, we consider the case $n=2$ only.
We start from  the resonator $B\e\cong \eps B$  connected to the outer space through the window
$$D\e=\left\{x=(x^1,x^2)\in\R^2:\ |x^1|< {d\e/2},\ x^2=0\right\}.$$ 
We introduce the domain $O\e$ consisting of two unit half-discs coupled via
$D\e$, namely,
$$O\e \ceq O^+\cup O^-\cup D\e,\text{ where }O^\pm\ceq \B(1,0)\cap\left\{x=(x^1,x^2)\in\R^2:\ \pm x^2>0\right\},$$
and let $W\e$ be a harmonic in $O\e$ function satisfying the following condition on $\partial O\e$:
\begin{align*}
	\begin{array}{lll}
		W\e(x)=\pm {1\over 2},& x\in \partial O^\pm\text{ with }x^2\not=0&\quad\text{(curved parts of $\partial O^\pm$)},\\[2mm]
		\left.{\partial W\e\over\partial x^n}\right|_{x^n=\pm 0}=0,& x\in \partial O^\pm\setminus D\e\text{ with }x^2=0&\quad\text{(rest of $\partial O\e$)}.
	\end{array}
\end{align*}  
Then we define the conductivity $\mathscr{C}(D\e)$ of the window $D\e$ by 
\begin{gather}\label{D:cond} 
	\mathscr{C}(D\e)=
	\|\nabla W\e\|_{\L(O\e)}^2.
\end{gather}
By using the properties \eqref{H:symm}--\eqref{H:n2} below one can easily show
that 
\begin{gather}\label{WH}
W\e(x)=\pm (1- H\e(x))/2\text{ for }x\in O^\pm,
\end{gather}
where $H\e$ is the solution to  \eqref{BVP:n2} (with $D\e$ instead of $D\ke$ and $z_k=0$).
It follows from \eqref{capty}, \eqref{D:cond}, \eqref{WH} that
$
\mathscr{C}(D\e)=\capty(D\e)/4;
$
consequently, the formula for the resonant eigenvalue take the form
\begin{gather}\label{gamma:conduct}
\gamma=\lim_{\eps\to 0}{\mathscr{C}(D\e)\over |B\e|}.
\end{gather}

Now, we turn to the resonator $B\e\cong \eps B$, which is connected to the outer space through 
the channel  (see Figure~\ref{fig-Sch})
$$P\e=\left\{x=(x^1,x^2)\in\R^2:\ |x^1|< {A\e/2},\ |x^2|\le {L\e/2}\right\},$$
We choose  the passage  length $L\e$ and  
its cross-section diameter   $A\e$   as in \cite{Sch15}:
\begin{gather} \label{LALA}
 L\e\sim L\eps,\quad A\e\sim A\eps^3,\quad A,\,L>0.
\end{gather}
Similarly to \eqref{D:cond}, 
we define the conductivity $\mathscr{C}(P\e)$ of $P\e$   by
\begin{gather*} 
	\mathscr{C}(P\e)=
	\|\nabla \wt W\e\|_{\L(\wt O\e)}^2,
\end{gather*}
where the set $\wt O\e$ consists of two unit half-discs connected through  $P\e$, namely,
$$\wt O\e = \wt O^+\cup \wt O^-\cup P\e,\quad 
\wt O^\pm=\{x=(x^1,x^2)\in\R^2:\ (x^1,x^2\mp L\e/2)\in O^\pm\},$$
and $\wt W\e$ is a harmonic in $\wt O\e$ function being equal to $\pm{1\over 2}$
on the curved part of $\partial 
\wt O^\pm$, and the Neumann conditions
on the rest of $\partial \wt O\e$.  One can prove that under 
the assumptions \eqref{LALA}  the following holds true:
$$	\|\nabla \wt W\e\|_{\L(\wt O\e)}^2\sim \|\nabla \wt U\e\|_{\L(\wt O\e)}^2={A\e/L\e},\text{ where } \wt U\e(x^1,x^2)
\ceq
\begin{cases}
\pm1/2,&x\in \wt O^\pm,	\\
x^2/L\e,&x\in P\e.
\end{cases}
$$
Hence the formula 
 \eqref{gamma:Sch} for the resonant eigenvalue also takes the form  \eqref{gamma:conduct}, but
with $\mathscr{C}\e(P\e)$ instead of $\mathscr{C}\e(D\e)$.
\end{remark}
 
}

\section{{Abstract theorem}\label{sec:3}}

Let $(\HS\e)_{\eps>0}$ be a family of Hilbert spaces, 
$(\A\e)_{\eps>0}$ be a family of non-negative, self-adjoint, unbounded operators in $\HS\e$, $(\a\e)_{\eps>0}$ be a family of associated 
sesquilinear forms. Also, let $\HS$ be a Hilbert space, {$\A$} be a non-negative, self-adjoint, unbounded operator in $\HS$, and $\a$ be the associated 
sesquilinear form.  
Along with $\HS\e$ and $\HS$ we define the Hilbert spaces $\HS^1\e$ and $\HS^1$  (energetic space associated with the forms $\a\e$ and $\a$) via
\begin{equation}\label{scale:1}
	\begin{array}{ll}
		\HS^1\e=\dom(\a\e),&
		\|u\|_{\HS^1\e}^2=\a\e[u,u]+\|u\|^2_{\HS\e},\\[2mm]
		\HS^1=\dom(\a),&
		\|f\|_{\HS^1}^2=\a[f,f]+\|f\|^2_{\HS},
	\end{array}
\end{equation}{
and the Hilbert space   $\HS^2$ via
\begin{gather} \label{scale:2}
	\HS^2\ceq\dom(\A),\quad
	\|f\|_{\HS^2}\ceq\|{(\A+\Id) f}\|_{\HS}.
\end{gather} } 
 
{ 
The following theorem is a corollary of several abstract results from
\cite{KP21,P06,P12}. For the sake  of the reader's convenience we recall  these results later in
Appendix~\ref{appendix:A}, and then in Appendix~\ref{appendix:B} we show how they imply  Theorem~\ref{thAAA}. }

{ 

\begin{theorem}\label{thAAA}	
		Let $\J\e \colon \HS\to  \HS\e$, ${\wt\J\e }\colon {\HS\e}\to \HS$ be linear bounded operators satisfying
	\begin{align} 
		\label{AAA:1}   
		(u,\J\e f)_{\HS\e} - (\wt\J\e u,f)_{\HS} =0,&&  
		\forall f\in\HS,\, u\in\HS\e,
		\\[1mm]
		\label{AAA:2}
		\|f-\wt\J\e \J\e f \|_{\HS} \leq \delta\e\|f\|_{\HS^1},&&  \forall f\in\HS^1,
		\\[1mm]
		\label{AAA:3}
		\|u-\J\e\wt \J\e u \|_{\HS\e}\leq \delta\e\|u\|_{\HS\e^1},&&   \forall u\in \HS^1\e, 
		\\[1mm] 
		\label{AAA:4} 
		\|\J\e f \|_{\HS\e} \leq 2\|f\|_{\HS},&& \forall f\in \HS,
	\end{align}
	where $\delta\e\in (0,1/2]$, $\delta\e\to 0$ as $\eps\to 0$.
	Furthermore, let 
	$\J\e^1 \colon \HS^1\to  \HS\e^1$, ${\wt\J\e^1 }\colon {\HS^1\e}\to \HS^1$
	be linear operators satisfying
	\begin{align}
		\label{AAA:5}
		\|\J\e^1 f-\J\e f\|_{\HS\e}&\leq \delta\e\|f\|_{\HS^1 },
		&& \forall f\in \HS^1,
		\\[1mm]
		\label{AAA:6} 
		\|\wt\J\e^1 u - \wt\J\e u \|_{\HS}&\leq  \delta\e\|u\|_{ \HS^1\e},&&\forall u\in\HS^1\e,
		\\[1mm]
		\label{AAA:7}
		|\a\e[\J\e f,u]-\a[f,\J'\e u]  |&\leq 
		\delta\e\|f\|_{\HS^2 }\|u\|_{\HS^1\e},&& \forall f\in \HS^2 ,\ u\in \HS^1\e.
	\end{align}
	Then the following properties hold true:
	\smallskip
\begin{itemize}
	\item[\rm(i)]  
	$
	\sigma(\A\e)\to \sigma(\A)$ and $
	\sigma_\disc(\A\e)\nearrow\sigma_\disc(\A)
	$ as $\eps\to 0.$\smallskip
	
	\item[\rm(ii)]
	The multiplicity is preserved in the following sense: if $\lambda\in\sigma_\disc(\A)$ is of multiplicity $\mu$ and
	$[\lambda-L,\lambda+L]\cap\sigma(\A)=\{\lambda\}$ with $L>0$, then 
	for sufficiently small $\eps$  the spectrum of  
	$\A\e$ in $[\lambda-L,\lambda+L]$ is purely discrete 
	and the total multiplicity of the eigenvalues of $\A\e$ contained in $[\lambda-L,\lambda+L]$ equals $\mu$.\smallskip
	
	\item[\rm(iii)]
	If, in addition, $\mu=1$ (i.e.  the eigenvalue $\lambda$ is simple), and  $\psi$ is the corresponding normalized in $\HS$ eigenfunction,  
	then there exists a sequence of   normalized in $\HS\e$ eigenfunctions $\psi\e$ of $\A\e$  such that $\|\psi\e - \J\e \psi \|_{\HS\e}\to 0\text{ as }\eps\to 0.$
	
	\item[\rm(iv)] One has the  estimate
   $\widetilde{d}_H\left(\sigma(\A\e),\,\sigma(\A)\right)\leq 10\delta\e.$
\end{itemize}	
	
\end{theorem}	 }

\section{Proof of the main results\label{sec:4}}

Recall  that $\B(\ell_k ,0)$ is the smallest ball  containing
$D_k$,
$( \rho_k)_{k\in\M}$ are positive numbers satisfying \eqref{DB1}, \eqref{DB2}, 
$(z_k)_{k\in\M}$ are points 
standing at the definitions of  $B\ke$ and $D\ke$. Evidently, 
\begin{gather}\label{elleps}
\B(\ell_k d\ke,z_k)\text{ is the smallest ball containing }D\ke.
\end{gather}
Further, we denote
\begin{align*}
R_k&\ceq {1\over 2}\min\left\{\dist(z_k,\cup_{l\not=k}\{z_l\}); \dist(z_k,\partial\Omega) \right\},\\ 
\tau_k&\ceq \inf\left\{\tau>0:\ B_k\cup\B(\rho_k,0)\subset \B(\tau,0)\right\}\\
&\,=
\inf\left\{\tau>0:\ B\ke\cup\B(\rho_k\eps,z_k)\subset \B(\tau\eps,z_k)\right\}.
\end{align*}
It is easy to see that
\begin{equation}\label{BRk:prop}
\overline{\B({R_k },z_k)}\subset\Omega\quad\text{and}\quad
{\B({R_k },z_k)}\cap{\B({R_l },z_l)}=\emptyset\text{ if } k\not=l.
\end{equation}
We also introduce the sets (below $z_k^n$ stands for the $x^n$-\,coordinate of $z_k$)
\begin{align}
\label{Bke+}
 B^+\ke&\ceq \B(\rho_k\eps,z_k)\cap\left\{x=(x',x^n)\in\R^n:\ x^n>z_k^n\right\},
\\\label{Bke-}
B^-\ke&\ceq \B(\rho_k\eps,z_k)\cap\left\{x=(x',x^n)\in\R^n:\ x^n<z_k^n\right\}.
\end{align}

In the following
we assume that $\eps\in (0,\eps_0]$ 
with $\eps_0>0$ satisfying
\begin{gather}
\label{tauR}
\eps_0<\min\left\{1, \min_{k\in\M}{R_k\over\tau_k}\right\},
\\
\label{CC}
\sup_{\eps\in(0,\eps_0]}C\e<\infty,
\\
\label{deps<1}
\sup_{\eps\in(0,\eps_0]} {d\e\over \eps}<1,
\end{gather}
where $C\e=\mathcal{O}(1)$ is a constant standing in \eqref{gamma:2};
the existence of $\eps_0$ satisfying \eqref{deps<1} is guaranteed by \eqref{gamma:2}.
We also denote
\begin{gather}\label{kappa}
\wt\tau_k\ceq{\tau_k\eps_0+R_k\over 2\eps_0}.
\end{gather}
It follows from the definitions of  $\ell_k$, $\rho_k$, $\tau_k$, $\wt\tau_k$, $R_k$,
and \eqref{DB2}, \eqref{deps}, \eqref{tauR} that
\begin{gather}\label{all:paramaters}
\forall k\in\M\ \forall\eps\in (0,\eps_0]:\quad
\ell_k d\ke<\ell_k\eps<\rho_k\eps\le\tau_k\eps<\wt\tau_k\eps <R_k.
\end{gather}
Note that $\rho_k=\tau_k$ only if $B\ke=B\ke^-$.
The inequalities \eqref{all:paramaters} (together with the properties \eqref{DBeps}, \eqref{elleps}) are clarified on Figure~\ref{fig-circles}.

\begin{figure}[h]
\begin{center}
\begin{picture}(250,250)
\scalebox{0.65}{\includegraphics{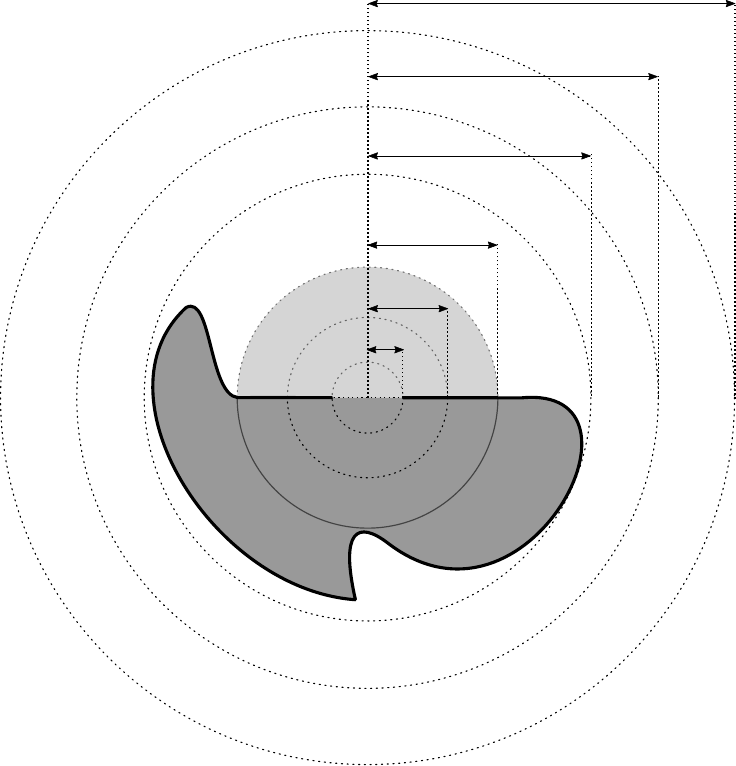}}

\put(-60,243){$_{R_k}$}
\put(-78,220){$_{\wt\tau_k\eps}$}
\put(-85,195){$_{\tau_k\eps}$}
\put(-100,168){$_{\rho_k\eps}$}
\put(-108,148){$_{\ell_k\eps}$}
\put(-115,135){$_{\ell_k d\ke}$}

\end{picture}
\end{center}
\caption{The domain $B\ke$ (dark gray color). The  part of $\partial B\ke$ corresponding to $S\ke$ is drawn via solid bold line, the part corresponding to $D\ke$ is drawn via dotted line.
The light gray half disc of the radius $\rho_k$ corresponds to $B\ke^+$; 
the corresponding lower half disc (its boundary is drawn via solid line) corresponds to $B\ke^-$. All circles has their center at $z_k$.} 
\label{fig-circles}
\end{figure}
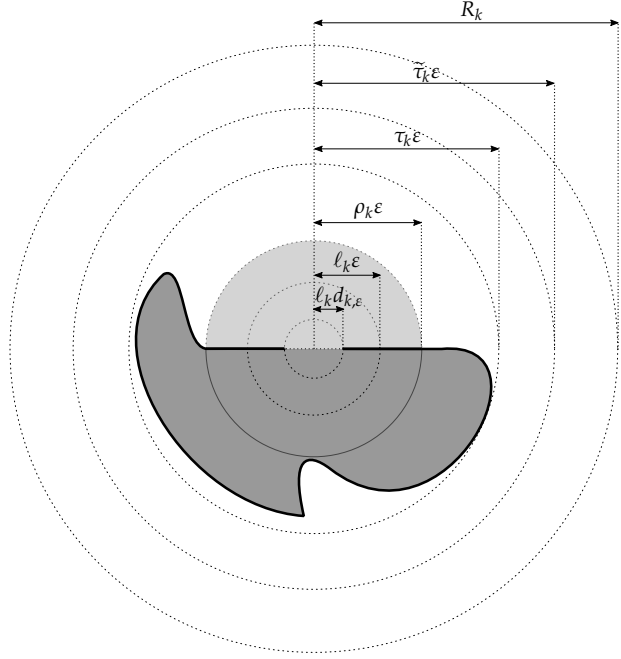

We denote
$$\HS\e\ceq \L(\Omega\e),\qquad\HS\ceq \L(\Omega)\oplus \C^m.$$ 
Recall that $\A\e$ and $\A$ are the self-adjoint non-negative operators 
in $\HS\e$ and $\HS$, respectively, being
associated with the forms $\a\e$ and $\a$ defined by
\eqref{ae} and \eqref{a}. We introduce the spaces $\HS\e^1$, $\HS^1$, $\HS^2$ as in \eqref{scale:1} and \eqref{scale:2}, i.e.
\begin{gather*} 
\begin{array}{ll}
\HS^1\e=\H^1_{0,\partial\Omega}(\Omega\e),
&
\ds\|u\|^2_{\HS^1\e}=\|u\|^2_{\H^1(\Omega\e)},\\[1.5ex]
\HS^1=\H_0^1(\Omega)\times\C^m,
&
\ds\|f\|^2_{\HS^1}=\|f_0\|^2_{\H^1(\Omega)}+
\sum_{k\in\M}(\gamma_k+1)|f_k|^2,
\\[1ex]
\HS^2=\dom(\Delta_\Omega^D)\times \C^m,
&\ds
\|f\|_{\HS^2}^2=
\|-\Delta f_0+f_0\|^2_{\L(\Omega)}+\sum_{k\in\M}(\gamma_k+1)^2|f_k|^2.
\end{array}
\end{gather*} 
Hereinafter, $f_0$ and $(f_k)_{k\in\M}$ stand, respectively, for 
 $\L(\Omega)$  and  $\C^m$ components of $f\in\HS$.

Our goal is to construct suitable operators 
$$\J\e     : \HS    \to \HS\e,\quad  
\wt\J\e  : \HS\e  \to \HS, \quad
\J\e^1   : \HS^1  \to \HS^1\e, \quad
\wt\J\e^1: \HS\e^1\to \HS^1,$$
satisfying the assumptions of {abstract Theorem~\ref{thAAA}}.
{We introduce the operators $\J\e$ and $\wt\J\e$  in Subsection~\ref{subsec4:1}, and 
summarize their properties in Lemma~\ref{lm:JJ:1}. In Subsection~\ref{subsec4:2} we
define the operators  $\J\e^1$   and  $\wt\J\e^1$,
whose properties are given in Lemmata~\ref{lm:JJ:2}, \ref{lm:JJ:5}, \ref{lm:JJ:6}.
The proof of Theorems~\ref{th1}, \ref{th2} is completed in Subsection~\ref{subsec4:3}.}
\smallskip

The following lemma will be frequently used further.
We denote
\begin{gather*}
\eta\e\ceq
\begin{cases}
\eps^2,& n\ge 3 
\\
\eps^2|\ln\eps|, & n=2 .
\end{cases}
\end{gather*}
Recall that ${B\ke\cup B\ke^+ }\subset   \B(R_k,z_k)$   (see \eqref{all:paramaters} and Figure~\ref{fig-circles}).

\begin{lemma}\label{lm:Best}
One has
\begin{gather}\label{lm:Best:est}
\forall u\in \H^1(\B(R_k,z_k)):\quad
\|u\|^2_{\L(B\ke\cup B\ke^+)}\leq C\eta\e\|u\|^2_{\H^1(\B(R_k,z_k))}
\end{gather}
\end{lemma}
 
\begin{proof}
Evidently, it is enough to prove \eqref{lm:Best:est} for $u\in \mathsf{C}^\infty(\overline{ \B(R_k,z_k) })$.
One has: 
\begin{align}\label{lm:Best:est1}
\|u\|_{\L(B\ke\cup B\ke^+)}^2&\leq
\|u\|_{\L(\B(\tau_k \eps,z_k))}^2\\&\leq 
C\left(\eps\|u\|^2_{\L(\partial(\B(\tau_k \eps,z_k)))}+\eps^2\|\nabla u\|^2_{\L(\B(\tau_k \eps,z_k))}\right), \notag 
\end{align}
where the first inequality is due to
$B\ke\cup B\ke^+ \subset\B(\tau_k \eps,z_k)$ (see~\eqref{all:paramaters}  and Figure~\ref{fig-circles}), while 
the second inequality  is deduced from the  estimate
\begin{gather*}
\|v\|_{\L(\B(\tau_k,0))}^2\leq 
(\lambda_1(-\Delta^R_{\B(\tau_k,0)}))^{-1}\left(\|v\|^2_{\L(\partial(\B(\tau_k,0)))}+\|\nabla v\|^2_{\L(\B(\tau_k ,0))}\right)\label{Robin}
\end{gather*}
via  the coordinate transformation $ y= (x-z_k)\eps^{-1}$, where $y\in \B(\tau ,0)$, $x\in \B(\tau\eps ,z_k)$; here $\Delta^R_{\B(\tau_k,0)}$
is  the Laplacian on $\B(\tau_k,0)$ subject to the Robin boundary conditions 
${\partial_n u} +u=0$.

We introduce the spherical coordinate system $(r,\phi)$ in $\overline{\B(R_k,z_k)\setminus \B(\tau_k\eps ,z_k)}$.
Here $r\in [\tau_k\eps,R_k]$ stands for the distance to $z_k$,
$\phi=(\phi_1,{\dots},\phi_{n-1})$ are the angular coordinates,  where $\phi_j\in [0,\pi]$ for $j=1,\dots,n-2$, 
$\phi_{n-1}\in [0,2\pi)$.
One has
$$
u(\tau_k\eps,\phi)=u(r,\phi)-\int_{\tau_k\eps}^{r}{\partial
u(\tau,\phi)\over\partial\tau}\d\tau,
$$
whence
\begin{align} \label{lm:Best:est2}
|u(\tau_k\eps,\phi)|^2&\le 2|u(r,\phi)|^2+2\left|\int_{\tau_k\eps}^{r}{\partial
u(\tau,\phi)\over\partial\tau}\d\tau\right|^2\\
\notag
&\leq
2|u(r,\phi)|^2+
2M\e \int_{\tau_k\eps}^{R_k}\left|{\partial
u(\tau,\phi)\over\partial\tau}\right|^2\tau^{n-1}\d\tau,
\end{align}
where $M\e\ceq \ds\int_{\tau_k\eps}^{R_k}\tau^{1-n}\d\tau$.
We denote $N\e\ceq \ds \int_{\tau_k\eps}^{R_k} r^{n-1}\d r $.
Multiplying \eqref{lm:Best:est2} by $(N\e)^{-1}(\tau_k\eps)^{n-1}r^{n-1}\prod_{j=1}^{n-2}\left(\sin\phi_j\right)^{n-1-j}$, integrating
over $r\in (\tau_k\eps,R_k)$, $\phi_j\in (0,\pi)$, $j=1,\dots,n-2$, 
$\phi_{n-1}\in (0,2\pi)$, { and taking into account that 
the left-hand-side   and the integral in the right-hand-side of \eqref{lm:Best:est2} are both independent on $r$,
we get
\begin{align}\label{spherical}
\int_{\mathbf{S}^{n}}|u(\tau_k\eps,\phi)|^2 (\tau_k\eps)^{n-1} \d \phi
&\leq
2N\e^{-1}(\tau_k\eps)^{n-1}\int_{\tau_k\eps}^{R_k}\int_{\mathbf{S}^{n}}|u(r,\phi)|^2 r^{n-1}\d\phi\d r 
\\\notag
&+
2M\e (\tau_k\eps)^{n-1} \int_{\tau_k\eps}^{R_k}\int_{\mathbf{S}^{n}}\left|{\partial
	u(\tau,\phi)\over\partial\tau}\right|^2\tau^{n-1}\d\phi \d\tau.
\end{align}
Here  $\d\phi$ stands for $\prod_{j=1}^{n-2}\left(\sin\phi_j\right)^{n-1-j}\d\phi_1\dots\d\phi_{n-1}$, which is nothing but
the surface element on the unit sphere $\mathbf{S}^{n}$ in $\R^n$; consequently,
 $(\tau_k\eps)^{n-1} \d \phi$ is the surface element on the sphere $\partial(\B(\tau_k\eps,z_k))$,
 and $r^{n-1} \d\phi\d r$ (or $\tau^{n-1} \d\phi\d\tau$) is the volume element.
Thus, \eqref{spherical} is equivalent to}
\begin{align}\label{lm:Best:est3}
\|u\|^2_{\L(\partial(\B(\tau_k\eps,z_k)))}
&\leq 
2(\tau_k\eps)^{n-1}\left((N\e)^{-1}\|u\|^2_{\L(T\ke)}
+
M\e\int_{\tau_k\eps}^{R_k}\int_{\mathbf{S}^{n}}\left|{\partial
	u(\tau,\phi)\over\partial\tau}\right|^2\tau^{n-1}\d\phi \d\tau\right)
\\\notag
&\leq 2(\tau_k\eps)^{n-1}\left((N\e)^{-1}\|u\|^2_{\L(T\ke)}
 +
M\e\|\nabla u\|^2_{\L(T\ke )}\right).
\end{align}
where $T\ke\ceq \B(R_k,z_k)\setminus \overline{\B(\tau_k\eps ,z_k)}$; 
{the last inequality  follows immediately from the fact that  in the spherical coordinates $(r,\phi)$ the length
of the gradient vector $\nabla u$    equals
$(|\partial_r u|^2+\underset{\ge 0}{\underbrace{\dots}})^{1/2}$.}
One has 
\begin{gather}\label{MN:eps}
M\e\leq C\eps^{ -n}\eta\e,\quad N\e\geq C>0
\end{gather}
(the first inequality in \eqref{MN:eps} is obtained via straightforward calculations,
while the second follows from $\tau_k\eps\leq \tau_k\eps_0<R_k$).
Using \eqref{MN:eps}, 
we can extend \eqref{lm:Best:est3} as follows,
\begin{gather} \label{lm:Best:est4} 
\|u\|^2_{\L(\partial(\B(\tau_k\eps,z_k)))}
\leq C\left(\eps^{n-1}\|u\|^2_{\L(T\ke)}
+
\eps^{-1}\eta\e\|\nabla u\|^2_{\L(T\ke)}\right).
\end{gather}
Combining \eqref{lm:Best:est1} and \eqref{lm:Best:est4}, and taking into account that
\begin{gather*}
\eps^n\leq C\eps^2\text{ with }C=\eps_0^{n-2},\quad 
\eps^2=\eta\e\text{ if }n\ge 3,\quad
\eps^2\leq C \eta\e\text{ with }C=|\ln\eps_0|^{-1}\text{ if }n=2,
\end{gather*}
we arrive at the  required estimate \eqref{lm:Best:est}. The lemma is proven.  
\end{proof}

\subsection{Operators $\J\e$ and $\wt\J\e$ and their properties\label{subsec4:1}}
We denote
\begin{gather*}
\Omega_{\eps}^{\rm out}\ceq \Omega\setminus\overline{\cup_{k\in\M}B_{k,\eps}},
\end{gather*}
i.e. ${\Omega\e}=\Omega\e^{\rm out}\cup\left(\cup_{k\in\M}B\ke\right)\cup\left(\cup_{k\in\M}D\ke\right)$.
We define the operator $\J\e:\HS\to\HS\e$ via{
\begin{gather*}
(\J\e f)(x)\ceq
\begin{cases}
f_0(x),&x\in \Omega_{\eps}^{\rm out},
\\
|B_{k,\eps}|^{-1/2}f_k,&x\in B_{k,\eps},\ k\in\M,
\end{cases}
\end{gather*}}
where $f=(f_0,f_1,\dots,f_m)$ with $f_0\in\L(\Omega)$, $(f_k)_{k\in\M}\in\C^m$.

Further, we introduce the operator $\wt\J\e:\HS\e\to\HS$ acting on $u\in\L(\Omega\e)$ as follows:
\begin{gather}\label{wtJ}
(\wt\J\e u)_0{(x)}\ceq 
\begin{cases}
u(x),&x\in \Omega_{\eps}^{\rm out},
\\
0,&x\in \overline{\cup_{k\in \M}B_{k,\eps}},
\end{cases}\qquad
(\wt\J\e u)_k\ceq |B_{k,\eps}|^{1/2}\la u \ra_{B_{k,\eps}},\ k\in\M,
\end{gather}
where $\la u \ra_{B_{k,\eps}}$ is the mean value of $u$ in $B\ke$, i.e.
$$\la u \ra_{B_{k,\eps}}\ceq |B\ke|^{-1}\int_{B\ke}u(x)\d x .$$

The above defined operators satisfy the following properties.

\begin{lemma}\label{lm:JJ:1}
One has
\begin{align}
\label{JJ:1}
(u,\J\e f)_{\HS\e}-(\wt\J\e u,f)_{\HS}=0,&\quad\forall f\in\HS,\ u\in\HS\e,
\\
\label{JJ:2}
\|\J\e f\|_{\HS\e}\leq \|f\|_{\HS},&\quad\forall f\in\HS,
\\
\label{JJ:3}
\|f-\wt\J\e\J\e f\|_{\HS}\leq C \eta\e ^{1/2}\|f\|_{\HS^1},&\quad\forall f\in{\HS^1},
\\
\label{JJ:4}
\|u-\J\e\wt\J\e u\|_{\HS\e}\leq C{ \eps}\|u\|_{\HS\e^1},&\quad\forall u\in{\HS\e^1}.
\end{align}
\end{lemma}

\begin{proof} 
Properties \eqref{JJ:1}--\eqref{JJ:2}  follow immediately  from the definition of 
$\J\e$, $\wt\J\e$. {
Furthermore, we have
\begin{gather}\label{fJJf}
(\wt\J\e\J\e f)_0(x)=
\begin{cases}
	f_0(x),&x\in \Omega_{\eps}^{\rm out},
	\\
	0,&x\in \overline{\cup_{k\in \M}B_{k,\eps}},
\end{cases}\qquad
(\wt\J\e\J\e f)_k=f_k,\ k\in\M.
\end{gather}}
Using   {\eqref{fJJf} and} Lemma~\ref{lm:Best}, we obtain the estimate \eqref{JJ:3}:
\begin{align*}
\|f-\wt\J\e\J\e f\|^2_{\HS}&=
\|f_0-(\wt\J\e\J\e f)_0\|^2_{\L(\Omega)}+
\suml_{k\in\M}|f_k-(\wt\J\e\J\e f)_k|^2
\\
&
=\suml_{k\in\M}\|f_0\|^2_{\L(B_{k,\eps})}
\leq C \eta\e\suml_{k\in\M}\|f_0\|^2_{\H^1(\B(R_k,z_k))}\\&
\leq C \eta\e \|f_0\|^2_{\H^1(\Omega)}
\leq C \eta\e \|f\|_{\HS^1}^2.
\end{align*}
Finally, we prove \eqref{JJ:4}. 
For each $u\in\H^1(B\ke)$
one has the   Poincar\'{e} inequality
\begin{gather}\label{Poincare}
\|u-\langle u \rangle_{B\ke}\|^2\leq {1\over \lambda_2(-\Delta_{B\ke}^N)}\|\nabla  u\|^2_{\L(B\ke)}=
{\eps^2\over \lambda_2(-\Delta_{B_k}^N)}\|\nabla  u\|^2_{\L(B\ke)} 
\end{gather}
(the last equality follows from  $B\ke\cong\eps B_k$).
Using \eqref{Poincare}, we get the estimate \eqref{JJ:4}:
\begin{gather*}
\|u-\J\e\wt\J\e u\|^2_{\HS\e}=
\sum_{k\in\M}\|u- \la u \ra_{B\ke}\|^2_{\L(B\ke)}\leq 
C\eps^2\sum_{k\in\M}\|\nabla u\|^2_{\L(B\ke)}\leq C\eps^2 \|u\|_{\HS\e^1}^2.
\end{gather*}
The lemma is proven.
\end{proof}

\subsection{Operators $\J^1\e$ and $\wt\J^1\e$ and their properties\label{subsec4:2}}

To construct the operator $\wt\J^1\e$, one needs the following auxiliary result.

\begin{lemma}[{\cite[Lemma~2.6]{ACDP92}}]\label{lm:ext}
Let $Y$ be a bounded domain in $\R^n$. 
Let $Y^* $ be a domain in $\R^n$ having Lipschitz boundary
at each point of $\partial Y^*\cap\overline{Y}$.
Then  there exists a linear operator
$P: \H^1(Y^*)\to \H^1(Y)$ such that for each $u\in \H^1(Y^*)$
one has
\begin{gather}\label{Pext}
\begin{array}{l}
P u=u\ \text{ a.e. in }Y\cap Y^*,\\[2mm]
\|P u \|_{\L(Y )}\leq C_1\| u\|_{\L(Y^*)},\quad 
\|\nabla(P u )\|_{\L(Y )}\leq C_2\|\nabla u\|_{\L(Y^*)},
\end{array}
\end{gather}
where the constants $C_1,C_2$  depend only on $Y$ and $Y^*$.
\end{lemma}

Owing to the fact that in Lemma~\ref{lm:ext} one has separate estimates 
for $u$ and $\nabla u$, we immediately get the following corollary.

\begin{corollary}\label{coro:ext}
Let $Y,Y^*\subset\R^n$ be  as in Lemma~\ref{lm:ext},  and $P: \H^1(Y^*)\to \H^1(Y)$ be the operator satisfying \eqref{Pext}. 
Let $Y\e\ceq\eps Y+z$, $Y^*\e\ceq \eps Y^*+z$, where $\eps>0$, $z\in\R^n$. 
We define the operator
$$
P\e: \H^1(Y\e^*)\to \H^1(Y\e),\quad
(P\e u)(x)= (P v\e)((x-z)\eps^{-1}),\ x\in Y\e,
$$
where $v\e(y)= u(y\eps+z)$, $y\in Y^*$.
Then  for each $ u\in \H^1(Y^*\e)$ one has
\begin{gather*}
\begin{array}{l}
P\e u=u\ \text{ a.e. in }Y\e\cap Y\e^*,\\[2mm]
\|P\e u \|_{\L(Y\e )}\leq C_1\| u\|_{\L(Y\e^*)},\quad 
\|\nabla(P\e u )\|_{\L(Y\e )}\leq C_2\|\nabla u\|_{\L(Y^*\e)},
\end{array}
\end{gather*}
where the constants $C_1,C_2$ are the same constants as in \eqref{Pext}.
\end{corollary}

One has $B\ke=\eps B_k+z_k$ and $\B(\wt\tau_k \eps,z_k)=\eps\B(\wt\tau_k,0)+z_k$
(recall that $\wt\tau_k$ is defined by \eqref{kappa}).
Since $B\ke\subset\B(\tau_k,z_k)$ and $\tau_k<\wt\tau_k$, we have $\overline{B\ke}\subset
\B(\wt\tau_k \eps,z_k)$ (that is, $\B(\wt\tau_k\eps,z_k)\setminus \overline{B\ke}$ is an open connected set with Lipschitz boundary).
By Corollary~\ref{coro:ext} there exists 
a linear operator 
$$P\ke:\H^1(\B(\wt\tau_k\eps,z_k)\setminus \overline{B\ke})\to \H^1(\B(\wt\tau_k\eps,z_k))$$ 
such
that
\begin{gather}\label{Pk:ext}
\begin{array}{l}
P\ke u=u\ \text{ a.e. in }\B(\wt\tau_k\eps,z_k)\setminus \overline{B\ke},\\[2mm]
\|P\ke u \|_{\H^1(\B(\wt\tau_k\eps,z_k))}\leq C_k\| u\|_{\H^1(\B(\wt\tau_k\eps,z_k)\setminus \overline{B\ke})}.
\end{array}
\end{gather} 
Finally, we define   the operator    $\wt\J\e^1:\HS\e^1\to\HS^1 $ by
\begin{gather}\label{wtJ1}
\begin{array}{l}
(\wt\J\e^1 u)_0(x)\ceq 
\begin{cases}
(P\ke(u\restr_{\B(\wt\tau_k\eps,z_k)\setminus \overline{B\ke}}))(x),&x\in \B(\wt\tau_k\eps,z_k),\ k\in\M,\\
u(x),&x\in\Omega\setminus\left(\cup_{k\in\M}\B(\wt\tau_k\eps,z_k)\right),
\end{cases} \\[5mm]   
(\wt\J\e^1 u)_k\ceq (\wt\J\e  u)_k,\ k\in\M.
\end{array}
\end{gather}
It follows easily from \eqref{Pk:ext}, \eqref{BRk:prop} and the last inequality in \eqref{all:paramaters} that
$(\wt\J\e^1 u)_0 \in\H^1_0(\Omega)$; thus \eqref{wtJ1} provides a well-defined linear operator $\wt\J\e^1:\HS^1\e\to\HS^1$ such that
\begin{gather}
\label{E}
(\wt\J\e^1 u)_0 = u\ \text{ a.e. on }\Omega\e^{\rm out},\quad 
\|(\wt\J\e^1 u)_0 \|_{\H^1(\Omega)}\leq C\| u\|_{\H^1(\Omega\e^{\rm out})}.
\end{gather}

\begin{lemma}\label{lm:JJ:2}
 One has
\begin{align}
\label{JJ:6}
\|\wt\J^1\e u - \wt\J\e u\|_{\HS }\leq C{\eta\e^{1/2}}\|u\|_{\HS^1\e},&\quad\forall u\in\HS^1\e.
\end{align}
\end{lemma}

\begin{proof}
Let $u\in \HS^1\e$. Using Lemma~\ref{lm:Best} and \eqref{E} we get:
\begin{align*} 
\|\wt\J^1\e u - \wt\J\e u\|_{\HS }^2&=
\suml_{k\in\M}\|(\wt\J^1\e u)_0\|^2_{\L(B\ke)}
\leq 
C \eta\e\suml_{k\in\M}\|(\wt\J^1\e u)_0 \|^2_{\H^1(\B(R_k,z_k))}
\\&\leq
C \eta\e \|(\wt\J^1\e u)_0 \|^2_{\H^1(\Omega)}
\leq 
C_1 \eta\e\|u\|^2_{\H^1(\Omega\e^{\rm out})}
\leq 
C_1 \eta\e\|u\|^2_{\HS\e^1}.
\end{align*}
The lemma is proven.
\end{proof}

Now, we introduce the most tricky operator $\J\e^1:\HS^1\to\HS\e^1$.
Recall that $H\ke(x)$ is a  solution to the problem
\eqref{BVP:n3} if $n\ge 3$ or to the problem 
\eqref{BVP:n2} if $n= 2$; in the latter case we extend $H\ke(x)$  by $0$ to the whole $\R^2$.
Recall that the sets $B\ke^\pm$  are defined by \eqref{Bke+}--\eqref{Bke-}.

For  $f=(f_0,f_1,\dots,f_m)\in\HS^1$ with $f_0\in\H^1_0(\Omega)$ and $ f_k \in\C$, ${k\in\M}$  
we set
\begin{gather}
\label{J1}
(\J\e^1 f)(x)
\ceq
\begin{cases}\ds
f_0(x)-\suml_{k\in\M}\wh\phi\ke(x)f_0(x)\\ 
\ds +{1\over 2}H\ke(x)\phi\ke(x)
|B_{k,\eps}|^{-1/2}f_k,&x\in\Omega\e^{\rm out},\\\ds
\left(1-{1\over 2} H\ke(x)\phi\ke(x)\right)|B_{k,\eps}|^{-1/2}f_k,&x\in B\ke,\, k\in\M,\\\ds
{1\over 2}|B\ke|^{-1/2}f_k,&x\in D\ke,\, k\in\M.
\end{cases}
\end{gather}
Here $\wh\phi\ke$ and $\phi\ke$ are cut-off functions defined as follows:

\begin{itemize}
\item $\wh\phi\ke(x)=\begin{cases}
1,&|x-z_k|\le \ell_k d\ke, \\
\ds\frac
{\mathcal{G}(|x-z_k|)-\mathcal{G}(\ell_k\eps)}
{\mathcal{G}(\ell_k d\ke)-\mathcal{G}(\ell_k\eps)},
&\ell_k d\ke< |x-z_k|<\ell_k\eps, \\
0,&|x-z_k|\ge \ell_k\eps,
\end{cases}
$  \\[3pt]
where   the function $\mathcal{G}:(0,\infty)\to\R$ is given by
$$
\mathcal{G}(t)\ceq 
\begin{cases}
t^{2-n},&n\ge 3,\\
\ln t,& n=2,
\end{cases}
$$

\item 
$\phi\ke(x)=\phi\left({|x-z_k|\over \eps}\right)$, where $\phi:(0,\infty)\to\R$ is a smooth function
satisfying
\begin{gather*}
\phi(t)=1\text{ if }t\le \ell_k,\qquad \phi(t)=0\text{ if }t\ge \rho_k.
\end{gather*}
\end{itemize}
It is easy to see that 
$(\J\e^1 f)\restr_{\Omega_{\eps}^{\rm out}}\in\H^1(\Omega_{\eps}^{\rm out})$ and 
$(\J\e^1 f)\restr_{B\ke}\in \H^1(B\ke)$ for each $k\in\M$.
Moreover, the traces of 
 $(\J\e^1 f)\restr_{B_{k,\eps}}$  and $(\J\e^1 f)\restr_{\Omega_{\eps}^{\rm out}}$ 
on $D\ke$  coincide -- they both equal to the constant ${1\over 2}|B\ke|^{-1/2}f_k$.
Hence  $\J\e^1 f\in\H^1(\Omega\e)$. Finally, since $\J\e^1 f=f_0$ on
$\Omega\setminus \overline{\cup_{k\in\M}\B(R_k,z_k)}$, one gets   $\J\e^1 f\in \H^1_{0,\partial\Omega\e}(\Omega\e)$.
Thus \eqref{J1}  provides a well-defined linear operator 
from $\HS^1$ to $\HS\e^1$.

To proceed further, 
we need pointwise estimates on the function $H\ke$.
Recall that $\B(\ell_k d\ke,z_k)$ is the smallest ball containing $D\ke$.

\begin{lemma}[{\cite[Lemma~2.4]{MK06}}]
\label{lemma:Hest}
Let $x\in\R^n\setminus\overline{\B(\ell_k d\ke,z_k)}$ with 
\begin{gather}\label{C0}
\rho(x)\geq C_0 d\ke\text{ if }n\geq 3\text{\quad and\quad  }
\rho(x)\geq \exp(-C_0 {|\ln d\ke|^{1/2}})\text{ if n=2},
\end{gather} 
where $\rho(x)=|x-z_k|-\ell_k d\ke$ is the distance from $x$ to $\B(\ell_k d\ke,z_k)$, $C_0>0$ is a constant. 
Then
\begin{gather}\label{H:estimates}
\begin{array}{lll}
|H\ke(x)|\leq 
C\dfrac{ (d\ke)^{n-2}}{(\rho(x))^{n-2}},&
|\nabla H\ke(x)|\leq 
C\dfrac{ (d\ke)^{n-2}}{(\rho(x))^{n-1}}&\text{as }n\geq 3,\\[2ex]
|H\ke(x)|\leq 
C\dfrac{ |\ln d\ke|^{-1}}{|\ln  \rho(x) |^{-1} },&
|\nabla H\ke(x)|\leq 
C\dfrac{ |\ln d\ke|^{-1} }{  \rho(x) }&\text{as }n=2,
\end{array}
\end{gather}
where $C>0$ is independent on $d\ke$ and $\eps$ (but may depend on $C_0$).
\end{lemma} 

\begin{remark}\label{rem:Hest}
Let $x\in \R^n$ such that $|x-z_k|\ge \ell_k\eps$; then $x$ satisfies \eqref{C0} with
$$C_0=
\begin{cases}
\ell_k(\kappa^{-1}-1),&n\ge 3,\\
{\sup}_{\eps\in (0,\eps_0]}\big(C\e^{1/2}\eps|\ln (\ell_k\eps(1-\kappa))|\big),&n=2,
\end{cases}
$$
where $\kappa\ceq \sup_{\eps\in(0,\eps_0]} {d\e  \eps^{-1}}<1$ (see \eqref{deps<1}), $C\e>0$ is a constant standing in \eqref{gamma:2} and satisfying $\sup_{\eps\in (0,\eps_0]}C\e<\infty$ (see \eqref{CC}). 
\end{remark}

\begin{lemma}\label{lm:JJ:5}
One has 
\begin{align}
\label{JJ:5}
\|\J^1\e f - \J\e f\|_{\HS\e}\leq C\eta\e^{1/2}\|f\|_{\HS^1},&\quad\forall f\in\HS^1.
\end{align}
\end{lemma}

\begin{proof} 
Let $f=(f_0,f_1,\dots,f_m)\in\HS^1$ with $f_0\in \H^1_0(\Omega)$, $f_k\in\C$, $k\in\M$.
One has
\begin{gather*}
 \J^1\e f - \J\e f= 
 \begin{cases}\ds
\suml_{k\in\M}\left( -\wh\phi\ke(x)f_0(x) +{1\over 2}H\ke(x)\phi\ke(x)
|B_{k,\eps}|^{-1/2}f_k\right),&x\in\Omega\e^{\rm out},\\\ds
- {1\over 2} H\ke(x)\phi\ke(x) |B_{k,\eps}|^{-1/2}f_k,&x\in B\ke,\\\ds
{1\over 2}|B\ke|^{-1/2}f_k,&x\in D\ke,
\end{cases}
\end{gather*}
whence, using $|\wh\phi\ke|\leq C$ and   $|B\ke|=\eps^n|B_k|$,  we infer the estimate
\begin{gather}\label{JminusJ:2}
\|\J^1\e f - \J\e f \|^2_{\HS\e}\leq 
C\sum_{k\in\M}\left[\|f_0\|^2_{\L( B^+_{k,\eps})}
+\eps^{-n}|f_k|^2 \|H\ke\phi\ke  \|^2_{\L( \B(\rho_k \eps,z_k) )} \right].
\end{gather}
The first term in the right-hand-side of \eqref{JminusJ:2}
is estimated via Lemma~\ref{lm:Best}: 
\begin{gather}\label{poincare:f0}
\suml_{k\in\M}\|f_0\|^2_{\L( B^+_{k,\eps})}
\leq  
C \eta\e\suml_{k\in\M}\| f_0\|^2_{\H^1(\B(R_k,z_k))}
\leq
C \eta\e \|f_0\|^2_{\H^1(\Omega)}.
\end{gather}
Now, we estimate the second term.
One has the following  Poincar\'{e} inequality:
\begin{gather}\label{Poincare:0}
\forall v\in\H_0^1(\B(\rho_k \eps,z_k)):\quad
\|v\|^2_{\L(\B(\rho_k \eps,z_k))}\leq {\eps^2\over \lambda_1(-\Delta_{\B(\rho_k,0)}^D)} \|\nabla  u\|^2_{\L(\B(\rho_k \eps,z_k))}.
\end{gather}
Since   $\supp(\phi\ke) \subset\overline{\B(\rho_k \eps,z_k)}$,
one has $H\ke \phi\ke\in \H^1_0(\B(\rho_k \eps,z_k))$. 
Applying \eqref{Poincare:0} for $v\ceq H\ke \phi\ke$, and then using 
$|\phi\ke|\leq C$, $|\nabla\phi\ke|\leq C\eps^{-1}$ and
\eqref{capty},
we obtain
\begin{align}\label{Hphi:est}
\|H\ke \phi\ke\|^2_{\L( \B(\rho_k \eps,z_k) )}
&\leq C\eps^2\|\nabla(  H\ke \phi\ke )\|^2_{\L( \B(\rho_k \eps,z_k) )}
\\\notag
&\leq
C_1\left(\eps^2 \capty( {D\ke})+ 
\| H\ke\|^2_{\L(\supp(\nabla\phi\ke))}\right).
\end{align}
By virtue of \eqref{cap:rescaling1}--\eqref{gamma:2}, \eqref{CC} we have  
\begin{gather}\label{cap:est}
\capty( {D\ke})\leq C\eps^n.
\end{gather}
One has $\supp(\nabla\phi\ke)\subset \R^n\setminus\B(\ell_k\eps,z_k)$, whence
(see Remark~\ref{rem:Hest}) the estimates \eqref{H:estimates} fulfill for 
$x\in \supp(\nabla\phi\ke)$. Using these estimates, \eqref{gamma:2}, \eqref{CC},  and $|\supp(\nabla\phi\ke)|\leq C\eps^n$, we get
\begin{gather}\label{HL2} 
\|H\ke\|^2_{\L(\supp(\nabla\phi\ke))}
\leq 
C_1\eps^{4-n}\cdot\left.
\begin{cases}
{(d\ke)^{2n-4}},& n\ge 3 
\\
{|\ln\eps|^2}\cdot {|\ln d\ke|^{-2}}, & n=2 
\end{cases}\right\}\leq
C_2\eps^{n }\eta\e^2.
\end{gather}
The estimate \eqref{JJ:5} follows from \eqref{JminusJ:2}, \eqref{poincare:f0},
\eqref{Hphi:est}--\eqref{HL2} (in the case $n=2$ we also use $\eps^2\leq C \eta\e$, $C=|\ln\eps_0|^{-1}$ and $\eta\e<\eps$).
The lemma is proven.
\end{proof}

Before to proceed to the   last lemma we give several several further properties of the function $H\ke$.
Standard elliptic regularity theory yields
\begin{gather}
\label{H:regul}
H\ke\in \mathsf{C}^\infty(\mathbb{R}^n\setminus\overline{D\ke})\text{ if }n\geq 3
\quad\text{and}\quad 
H\ke\in \mathsf{C}^\infty(\B(1,z_k)\setminus\overline{D\ke})\text{ if }n=2. 
\end{gather}
Evidently, $H\ke$ is symmetric 
with respect to the hyperplane 
\begin{align}
\label{Gak}
\S_k\ceq \left\{x=(x',x^n)\in\R^n:\ x^n=z_k^n\right\},
\end{align}
i.e.
\begin{gather}
    \label{H:symm}
H\ke(x',z_k^n+\tau)=H\ke(x',z_k^n-\tau),\ \forall \tau>0.
\end{gather} 
From \eqref{H:regul}, \eqref{H:symm}  we deduce
\begin{gather}\label{H:n1}
\left.{\partial H\ke\over\partial x^n}\right|_{x^n=z^n_k+0}=\left.{\partial H\ke\over\partial x^n}\right|_{x^n=z^n_k-0}=0\text{ on } \S_k\setminus \overline{D\ke},\\
\label{H:n2}
\left.{\partial H\ke\over\partial x^n}\right|_{x^n=z^n_k+0}=-\left.{\partial H\ke\over\partial x^n}\right|_{x^n=z^n_k-0}\text{ on } D\ke.
\end{gather} 
Finally, using \eqref{BVP:n3}--\eqref{BVP:n2}
and the asymptotics (see Lemma~\ref{lemma:Hest})
\begin{gather}\label{decay}
| H\ke(x)|=\mathcal{O}(|x|^{2-n}),\quad
|\nabla H\ke(x)|=\mathcal{O}(|x|^{1-n})\quad\text{if }n\ge 3,
\end{gather}
we obtain
the  Green's identity
\begin{gather}\label{Green}
\int_{D\ke}\left(\left.{\partial H\ke\over\partial x^n}\right|_{x^n=z^n_k-0}-
\left.{\partial H\ke\over\partial x^n}\right|_{x^n=z^n_k+0}\right)\d x'=
\begin{cases}\ds
\|\nabla H\ke\|^2_{\L(\R^n)},&n\ge 3,\\\ds
\|\nabla H\ke\|^2_{\L(\B(1,z_k)},&n= 2,\\
\end{cases}
\end{gather}
where $\d x'$ stands for the area measure on $\partial D\ke$ (the decaying property \eqref{Green} is required
to guarantee that $\int_{\partial \B(R,0)}H\ke{\partial_n H\ke }\d s\to 0$
as $R\to \infty$).
From \eqref{capty}, \eqref{H:n2}, \eqref{Green} we infer
\begin{gather}\label{cap:D} 
\mathrm{cap}( {D\ke})=2\int_{D\ke}
\left.{\partial H\ke\over \partial x^n}\right|_{x^n=z_k^n-0}\d x'.
\end{gather}

\begin{lemma}\label{lm:JJ:6}
 One has  $\forall f\in \HS^2 ,\ u\in \HS^1\e$:
\begin{align}
\label{JJ:7}
\left|\a\e[u,\J^1\e f]-\a[\wt\J^{1}\e u,f]  \right| \leq 
C\left( \suml_{k\in\M}|\ga\ke-\ga_k|+(\eta\e)^{3/2}\eps^{-2}\right)\|f\|_{\HS^2 }\|u\|_{\HS^1\e}.
\end{align}
\end{lemma}

\begin{proof}
Let $f\in \HS^2$ and $u\in \HS^1\e$. We have
\begin{gather}\notag
\a\e[u,\J^1\e f]-\a[\wt\J^{1}\e u,f] = \suml_{k\in\M}\left(P\e^k+Q\e^k+R\e^k\right),
\end{gather}
where
\begin{align*}
P\e^k=&\,
-\suml_{k\in\M} \int_{ B^+_{k,\eps}}\nabla u\cdot \nabla\left(\wh\phi\ke \overline{f_0}\right)\d x,
\\
Q\e^k=&\,-\suml_{k\in\M} \int_{B\ke}\nabla (\wt\J^{1}\e u)\cdot \nabla \overline{f_0}\d x,
\\
R\e^k=&\,
  \suml_{k\in\M} 
 \left({1\over 2}\int_{ B^+\ke}\nabla u\cdot \nabla\left(H\ke\phi\ke
  \right) \d x\right. \\
  &\left.- {1\over 2}
\int_{B\ke^-}
\nabla u\cdot
\nabla \left(H\ke\phi\ke \right)\d x
-\gamma_k |B\ke| \la u \ra_{B\ke}\right) |B_{k,\eps}|^{-1/2}  \overline{f_k }.
\end{align*}

\paragraph{\it Estimate of $P\e^k$}
Using standard theory of elliptic PDEs (see, e.g., \cite{Ev98}), we infer
$\dom(\Delta_\Omega^D)\xhookrightarrow{}\H^2(\Sigma')$ 
for any bounded domain $\Sigma'\ssubset\Omega$. In particular,
there exists $C>0$ such that
\begin{gather}\label{h}
\forall h\in\dom(\Delta_\Omega^D):\quad
\|h\|_{\H^2(\B(R_k,z_k))}\leq C\|-\Delta h + h\|_{\L(\Omega)}.
\end{gather}
One has
\begin{gather}\label{P}
P\e^{k}=
-\suml_{k\in\M} \underbrace{\int_{ B^+_{k,\eps}}(\nabla u\cdot \nabla \overline{f_0})\wh\phi\ke\d x}_{P\e^{k,1}\ceq}
-\suml_{k\in\M} \underbrace{\int_{ B^+_{k,\eps}}(\nabla u\cdot \nabla \wh\phi\ke)
\overline{f_0}\d x}_{P\e^{k,2}\ceq}.
\end{gather}
Using Lemma~\ref{lm:Best}  and \eqref{h}, and taking into account that $|\wh\phi\ke|\leq C$,  we estimate
the first integral in \eqref{P} as follows:
\begin{align}\label{P1} 
|P\e^{k,1}|&\leq \|\nabla u\|_{\L( B^+_{k,\eps})}
 \|\nabla f_0\|_{\L( B^+_{k,\eps})}
  \leq 
C\eta\e^{1/2}\|\nabla u\|_{\L(\Omega\e)}
\|\nabla f_0\|_{\H^1(\B(R_k,z_k))} \\
\notag
&\leq C\eta\e^{1/2}\|\nabla u\|_{\L(\Omega\e)}\|f_0\|_{\H^2(\B(R_k,z_k))}\\
&\leq C_1\eta\e^{1/2}\|\nabla u\|_{\L(\Omega)}\|-\Delta f_0 + f_0\|_{\L(\Omega)}
\notag
\leq C_1\eta\e^{1/2}\|u\|_{\HS\e^1}\|f\|_{\HS^2}.
\end{align}
To estimate the second integral in \eqref{P} we use the generalised H\"older's inequality:
\begin{gather}\label{P2:start}
|P\e^{k,2}|\leq 
\|\nabla u\|_{\L( B^+_{k,\eps})} \|f_0\|_{\mathsf{L}^p( B^+_{k,\eps})}  
\|\nabla \wh \phi\ke\|_{\mathsf{L}^q( B^+_{k,\eps})},
\end{gather}
where $p,q\in [2,\infty]$ with ${1\over p}+{1\over q}={1\over 2}$ are chosen as follows:
\begin{gather}\label{pq}
\begin{array}{lll}
p={2n^2\over (n-2)^2},& q={n^2\over 2(n-1)}&\text{if }n\ge 3,
\\
p=\infty,& q=2&\text{if }n=2.
\end{array}
\end{gather}
{The Sobolev embedding theorem  \cite[Theorem~5.4]{Ad75} asserts that the space $\mathsf{L}^p(\B(R_k,z_k))$
is embedded continuously into the space $\H^2(\B(R_k,z_k))$ provided $p$ satisfies 
$1\leq p\leq \frac{2n}{n-4}$ if $n\geq 5$,
$1\leq p<\infty$ if $n=4$, and
$1\le p\le \infty$ if $n=2,3$; it is easy to check that 
$p$ defined by \eqref{pq} falls within the above restrictions.}
Using this and \eqref{h}, we get
\begin{align} \label{sobolev}
\|f_0\|_{\mathsf{L}^p( B^+\ke)}&\le 
\|f_0\|_{\mathsf{L}^p(\B(R_k,z_k))}\le 
C\|f_0\|_{\mathsf{H}^2(\B(R_k,z_k))}
\\&\leq
C_1\|-\Delta f_0+f_0\|_{\L(\Omega)}\leq
C_1\|f\|_{\HS^2}.\notag 
\end{align}
Furthermore, via direct calculations we obtain 
\begin{gather}\label{phi:est}
\|\nabla \wh \phi\ke\|_{\mathsf{L}^q( B^+\ke)}\leq
C\left.\begin{cases} 
(d\e)^{1-2/n},&n\ge 3
\\
|\ln d\e|^{-1/2},&n=2
\end{cases}\right\}\leq C_1\eps
\end{gather}
(the last inequality   is valid   by virtue of \eqref{gamma:2}, \eqref{CC}). 
It follows from \eqref{P2:start}, \eqref{sobolev}, \eqref{phi:est} 
that
\begin{gather}\label{P2}
|P\e^{k,2}|\leq
C\eps\|u\|_{\HS^1\e}\|f\|_{\HS^2}\leq 
C_1\eta\e^{1/2}\|u\|_{\HS^1\e}\|f\|_{\HS^2}.
\end{gather}
Combining \eqref{P1} and \eqref{P2}, we finally arrive at the estimate
\begin{gather}\label{P:final}
|P\e^{k}|\leq
C\eta\e^{1/2}\|u\|_{\HS^1\e}\|f\|_{\HS^2}.
\end{gather}
\medskip

\paragraph{\it Estimate of $Q\e^k$}
One has:
\begin{align}\label{Q:first}
|Q\e^k|^2&\leq
 {\suml_{k\in\M}\|\nabla (\wt J^1\e u)\|^2_{\L(B\ke)}}
 {\suml_{k\in\M}\|\nabla f_0\|^2_{\L(B\ke)}} 
 \\ &\leq
  \|\wt J^1\e u\|^2_{\H^1(\Omega)}
 {\suml_{k\in\M}\|\nabla f_0\|^2_{\L(B\ke)}}.\notag
\end{align}
Due to \eqref{E}, we have
\begin{gather}\label{Q1}
\|\wt\J\e^1 u\|^2_{\H^1(\Omega)}\leq 
C\| u\|^2_{\H^1(\Omega_{\eps}^{\rm out})}\leq 
C\|u\|^2_{\HS^1\e}.
\end{gather}
Lemma~\ref{lm:Best} and \eqref{h} yield
\begin{align}\label{Q2}
\|\nabla f_0\|_{\L( B_{k,\eps})}^2
& \leq 
C \eta\e\|\nabla f_0\|^2_{\H^1(\B(R_k,z_k))}
\leq C \eta\e \|f_0\|_{\H^2(\B(R_k,z_k))}\\ \notag 
&\leq C_1 \eta\e \|-\Delta f_0 + f_0\|^2_{\L(\Omega)}
\leq C_1 \eta\e \|f\|^2_{\HS^2}.
\end{align}
Combining \eqref{Q:first}--\eqref{Q2} we get the estimate
\begin{gather}\label{Q:final}
|Q\e^k|\leq C \eta\e^{1/2}\|u\|_{\HS\e^1}\|f\|_{\HS^2}.
\end{gather}

\paragraph{\it Estimate of $R\e^k$}
Recall that the hyperplane $\S_k$ is defined by \eqref{Gak}. 
The cut-off function $\phi\ke=\phi(|\cdot-z_k|\eps^{-1})\in C^\infty(\R^n)$  has the following properties:
\begin{gather}\label{phi:n}
\begin{array}{c}
\ds{\partial \phi\ke\over\partial x^n}=0\text{ on } \S_k,\\[2mm]
\supp(\phi\ke)\subset \overline{\B(\rho_k\eps,z_k)},\quad
\phi\ke=1\text{ in a neighbourhood of }D\ke.
\end{array}
\end{gather}
Then, using \eqref{H:n1} and \eqref{phi:n},
we obtain via partial integration:
\begin{multline}\label{byparts1}
\int_{ B^+\ke}\nabla u\cdot \nabla\left(H\ke\phi\ke\right) \d x\\
=
-\int_{ B^+\ke} u\cdot \Delta\left(H\ke\phi\ke\right) \d x -
\int_{D\ke} 
u\left.{\partial H\ke\over\partial x^n}\right|_{x^n=z^n_k+0}\d x'.
\end{multline}
Similarly, using \eqref{H:n1} and \eqref{phi:n}, we get
\begin{align}\label{byparts2}
&\int_{ B^-\ke}\nabla u\cdot \nabla\left(H\ke\phi\ke\right) \d x
\\\notag
=&
-\int_{  B^-\ke} u\cdot \Delta\left(H\ke\phi\ke\right) \d x +
\int_{D\ke} 
u\left.{\partial H\ke\over\partial x^n}\right|_{x^n=z^n_k-0}\d x'
\\\notag
=&-\int_{  B^-\ke} (u-\la u\ra_{B\ke})\cdot \Delta\left(H\ke\phi\ke\right) \d x 
\\\notag
-&
\la u\ra_{B\ke}\int_{  B^-\ke} \Delta\left(H\ke\phi\ke\right) \d x
+
\int_{D\ke} 
u\left.{\partial H\ke\over\partial x^n}\right|_{x^n=z^n_k-0}\d x'
\\\notag
=&
-\int_{  B^+\ke} (u-\la u\ra_{B\ke})\cdot \Delta\left(H\ke\phi\ke\right) \d x  
\\&-
\la u\ra_{B\ke}\int_{D\ke} 
\left.{\partial H\ke\over\partial x^n}\right|_{x^n=z^n_k-0}\d x'
+
\int_{D\ke} 
u\left.{\partial H\ke\over\partial x^n}\right|_{x^n=z^n_k-0}\d x'.\notag
\end{align}
Combining \eqref{byparts1} and \eqref{byparts2} and taking into account \eqref{H:n2} and \eqref{cap:D},
we get the equality
\small
\begin{multline*}
R\e^k=
  \suml_{k\in\M} 
 \underbrace{\left({1\over 4}\capty( {D\ke})
-\gamma_k |B\ke|\right) \la u \ra_{B\ke} |B_{k,\eps}|^{-1/2}  \overline{f_k }}_{R\e^{k,1}\ceq}\\
+{1\over 2}\suml_{k\in\M}
\underbrace{\left(-\int_{ B^+\ke} u\cdot \Delta\left(H\ke\phi\ke\right) \d x  
+\int_{  B\ke^-} (u-\la u\ra_{B\ke})\cdot \Delta\left(H\ke\phi\ke\right) \d x \right)|B_{k,\eps}|^{-1/2}  \overline{f_k}}_{R\e^{k,2}\ceq}
\end{multline*}\normalsize
Due to \eqref{gamma:1} and \eqref{wtJ}, one has 
$R\e^{k,1} = \left(\gamma\ke
-\gamma_k  \right) (\wt \J\e u)_k \overline{f_k }$, whence, taking into account
that $\|\wt\J\e u\|_{\HS}\leq \|u\|_{\HS\e}$ (this follows immediately from \eqref{JJ:1}--\eqref{JJ:2}),
we obtain
\begin{align}\label{R1:final}
|R\e^{k,1}|\leq  
C|\gamma\ke-\gamma_k|   \cdot \|\wt \J\e u\|_{\HS}\cdot\|f\|_{\HS}
&\leq 
C_1|\gamma\ke-\gamma_k|   \cdot \|u\|_{\HS\e}\cdot\|f\|_{\HS}
\\ \notag
&\leq 
C_1|\gamma\ke-\gamma_k|   \cdot \|u\|_{\HS\e^1}\cdot\|f\|_{\HS^2}.
\end{align}
It remains to estimate the term $R\e^{k,2}$.
One has, taking into account $B\ke^-\subset B\ke$:
\begin{align*}
|R\e^{k,2}|^2&\leq  
2\left(\|u\|_{\L( B^+\ke)}^2 +
\|u-\la u\ra_{B\ke}\|_{\L(B\ke)}^2\right)\times\\
&\times
\|\Delta\left(H\ke\phi\ke\right)\|_{\L(\B(\rho_k\eps,z_k))}^2
|B_{k,\eps}|^{-1}  |f_k|^2
\end{align*}
Applying Lemma~\ref{lm:Best} and \eqref{E},   we get
\begin{align} 
\|u\|_{\L( B^+\ke)}^2=
\|\wt\J\e^1 u\|_{\L( B^+\ke)}^2
\leq 
C \eta\e \|\wt\J\e^1 u\|^2_{\H^1(\B(R_k,z_k))}\leq
C_1 \eta\e\|u\|_{\HS\e^1}^2.
\label{u:whB}
\end{align}
Further, since
$\Delta H\ke=0$, one has
\begin{gather}\label{Leibniz}
\Delta\left(H\ke\phi\ke\right)=
2\nabla H\ke\cdot \nabla \phi\ke+
H\ke\Delta \phi\ke.
\end{gather}
One has
$\supp(\nabla\phi\ke)\cup\supp(\Delta\phi\ke)\subset \R^n\setminus\B(\ell_k\eps,z_k)$,  whence (see~Remark~\ref{rem:Hest})
the estimates \eqref{H:estimates} hold
for $x\in \supp(\nabla\phi\ke)\cup\supp(\Delta\phi\ke)$. Using them,  \eqref{gamma:2} and \eqref{CC}, we get
\begin{gather}\label{HgradH:est}
x\in\supp(\nabla\phi\ke)\cup\supp(\Delta\phi\ke):\quad
|(\nabla H\ke)(x)|\leq C\eps,\ 
|H\ke(x)|\leq 
C\eta\e.
\end{gather}
It follows from \eqref{Leibniz}, \eqref{HgradH:est},  
$|\nabla\phi\ke|\leq C\eps^{-1}$, $|\Delta\phi\ke|\leq C\eps^{-2}$,
$|\B(\rho_k\eps,z_k)|\leq C\eps^n$ that
\begin{gather}
\label{DeltaHphi:est}
\|\Delta\left(H\ke\phi\ke\right)\|_{\L(\B(\rho_k\eps,z_k))}^2
\leq C\eps^{n-4} \eta\e^2.
\end{gather}
Combining \eqref{Poincare}, \eqref{u:whB}, \eqref{DeltaHphi:est} and taking into account that
$|B\ke|= C\eps^n$, we get
\begin{gather}\label{R2:final}
|R\e^{k,2}|\leq C\eta\e^{3/2}\eps^{-2}\|u\|_{\HS^1\e}\|f\|_{\HS}\leq C\eta\e^{3/2}\eps^{-2}\|u\|_{\HS^1\e}\|f\|_{\HS^2}.
\end{gather}
Finally, from \eqref{R1:final} and \eqref{R2:final}, we conclude
\begin{gather}\label{R:final}
|R\e^{k}|\leq C\left(\suml_{k\in\M}|\ga\ke-\ga_k|+\eta\e^{3/2}\eps^{-2}\right)\|u\|_{\HS^1\e}\|f\|_{\HS^2}.
\end{gather}
The required estimate \eqref{JJ:7} follows from \eqref{P:final}, \eqref{Q:final}, \eqref{R:final}
(in the case $n=2$ we also use $\eta\e^{1/2}\leq C \eta\e^{3/2}\eps^{-2}$).
The lemma is proven.
\end{proof}

\subsection{End of proofs of Theorems~\ref{th1} and \ref{th2}\label{subsec4:3}}

{
It follows from  
\eqref{JJ:1}--\eqref{JJ:4}, \eqref{JJ:6}, \eqref{JJ:5},  \eqref{JJ:7}
that  conditions 
\eqref{AAA:1}--\eqref{AAA:7} of Theorem~\ref{thAAA} hold with 
$$\delta\e\ceq C\left(\suml_{k\in\M}|\ga\ke-\ga_k|+(\eta\e)^{3/2}\eps^{-2}\right)=\begin{cases}
\ds C \sum_{k\in\M}|\ga\ke-\ga_k|+C\eps ,&n\ge 3,\\
\ds C \sum_{k\in\M}|\ga\ke-\ga_k|+C\eps|\ln\eps|^{3/2} ,&n=2.
\end{cases}$$
(in the case $n=2$ we also use  
$\eps\leq C \eta\e^{3/2}\eps^{-2}$,
$\eta\e^{1/2}\leq C \eta\e^{3/2}\eps^{-2}$).
Then, by virtue of Theorem~\ref{thAAA} and the equality
$$
\|\psi\e - \J\e \psi  \|_{\HS\e}^2=
\|\psi\e - \psi_0\|^2_{\L(\Omega\e\setminus \overline{\cup_{k\in\M}B\ke})}+
\suml_{k\in\M}\|\psi\e - |B\ke|^{-1/2}\psi_k\|^2_{\L(B_{k,\eps})},$$
we immediately obtain  all statements of Theorems~\ref{th1} and \ref{th2}.
}

\subsection{Absence of outside convergence of discrete spectrum\label{subsec4:4}}
In this subsection we present an example showing that 
the property $\sigma_\disc(\A\e)\searrow\sigma_\disc(\A)$ does not always hold true.

Let   $\Omega=\Omega'\times\R$ be a tubular domain with a bounded Lipschitz
cross-section $\Omega'\subset \R^{n-1}$. To simplify presentation, we restrict ourselves to the case $n\ge 3$.   The spectrum  of the Dirichlet Laplacian 
$-\Delta_\Omega^D$ on $\Omega$ is purely essential, namely 
\begin{gather}\label{Delta:spec}
\sigma(-\Delta_\Omega^D)=\sigma_\ess(-\Delta_\Omega^D)=[\Lambda',\infty),
\end{gather}
where $\Lambda'\ceq \lambda_1(-\Delta^D_{\Omega'})>0$.
As before, we define the domain $\Omega\e$ by \eqref{Omega:e}, \eqref{DBe}, \eqref{Dk}
with $m=1$ (i.e., only one resonator is inserted), arbitrary sets $B_1\subset\R^n$ and $\wt D_1\subset\R^{n-1}$ such that \eqref{DB1}, \eqref{DB2} hold,
arbitrary $z_1\in\Omega$, and $d_{1,\eps}$ being chosen as follows:
\begin{gather*}
d_{1,\eps}=\left({4(\Lambda'-\eps^{1/2})|B_1|\over \mathrm{cap}(D_1)}\right)^{1\over n-2}\eps^{n\over n-2}.
\end{gather*}
With such a choice of $d_{1,\eps}$ we obtain, using \eqref{cap:rescaling2}:
\begin{gather}\notag
\gamma_{1,\eps}=\Lambda'-\eps^{1/2},\\\label{gaga2}
\text{consequently, }\gamma_1=\Lambda'.
\end{gather}
Let $\A$ be the corresponding limiting operator, see~\eqref{A:oplus}.
From  \eqref{A:oplus}, \eqref{Delta:spec}, \eqref{gaga2} we infer
\begin{gather}
\label{A:ess}
\sigma_\ess(\A)=[\Lambda',\infty),\quad
\sigma_\disc(\A)=\emptyset.
\end{gather}

Now, we look closely at the spectrum of $\A\e$. 
From \eqref{essAA} and \eqref{A:ess}  we deduce
\begin{gather}
\label{Ae:ess}
\sigma_\ess(\A\e)=[\Lambda',\infty).
\end{gather}
Furthermore, one can find such a function  
$v\e\in \mathrm{dom}(\a\e)$ that 
\begin{gather}\label{ve:RelQ}
{\a\e[v\e,v\e]\over \|v\e\|^2_{\L(\Omega\e)}}<
\Lambda'\text{ for sufficiently small $\eps$}.
\end{gather}
For convenience of presentation, we postpone 
the construction of $v\e$  and the justification of \eqref{ve:RelQ} 
to the end of this subsection.
Now, using  \eqref{Ae:ess}, \eqref{ve:RelQ} and
the min-max principle (see, e.g., \cite[Section~4.5]{D95}), we conclude that  
\begin{gather}
\label{Ae:disc}
[0,\Lambda')\cap\sigma_{\disc}(\A\e)\not=\emptyset \text{ for sufficiently small $\eps$.}
\end{gather}
From  \eqref{Ae:disc} we immediately infer that
there exists a sequence $(\lambda_{\eps_k})_{k\in\N}$ with $\lambda_{\eps_k}\in\sigma_{\disc}(\A_{\eps_k})$ converging to some   $\lambda_0\in [0,\Lambda']$ as $\eps_k\to 0$. By Theorem~\ref{th1} $\lambda_0\in\sigma(\A)$. On the other hand $\sigma(\A)\cap [0,\Lambda']=\{\Lambda'\}$, whence $\lambda_0=\Lambda'\in\sigma_{\ess}(\A)$, 
which means 
(cf.~Remark~\ref{rem:equiv}) that the property
$\sigma_\disc(\A\e)\searrow\sigma_\disc(\A)$ does not hold true.

It remains to construct the function $v\e$ satisfying \eqref{ve:RelQ}.
We choose it as follows: 
\begin{align*}
v\e=
\begin{cases}\ds
{1\over 2}H_{1,\eps}(x)\phi_{1,\eps}(x)
|B_{1,\eps}|^{-1/2},&x\in\Omega\e^{\rm out},\\\ds
\left(1-{1\over 2} H_{1,\eps}(x)\phi_{1,\eps}(x)\right)|B_{1,\eps}|^{-1/2},&x\in B_{1,\eps},\\\ds
{1\over 2}|B_{1,\eps}|^{-1/2},&x\in D_{1,\eps}, 
\end{cases} 
\end{align*}
i.e. $v\e=\J\e^1 g$, where 
$g=(0,1)\in\HS^1=\H^1_0(\Omega)\times\C$. 
It is easy to see that 
\begin{align*}
\a\e[v\e,v\e]
={1\over 4}|B_{1,\eps}|^{-1}  
\|\nabla (H_{1,\eps}\phi_{1,\eps})\|^2_{\L(\R^n )}
={1\over 4}|B_{1,\eps}|^{-1} \left(
\capty(D_{1,\eps})+U_{\eps,1}+U_{\eps,2}\right),
\end{align*}
where 
$$U_{\eps,1}\ceq\|\nabla (H_{1,\eps}(\phi_{1,\eps}-1))\|^2_{\L(\R^n )},\quad U_{\eps,2}\ceq 2(\nabla H_{1,\eps},\nabla (H_{1,\eps}(\phi_{1,\eps}-1)))_{\L(\R^n)}.$$
Taking into account that
$|\phi_{1,\eps}|\leq C$ and $|\nabla\phi_{1,\eps}|\leq C\eps^{-1}$,
we estimate $U_{1,\eps}$ as follows:
\begin{align}\label{inter:1}
\|\nabla (H_{1,\eps}(\phi_{1,\eps}-1))\|^2_{\L(\R^n )}\leq
C\left(\|\nabla H_{1,\eps}\|^2_{\L(\supp(\phi_{1,\eps}-1))}+
\eps^{-2}\|H_{1,\eps}\|^2_{\L(\supp(\nabla\phi_{1,\eps}))}\right).
\end{align}
One has (see~\eqref{HL2}; recall that we assume $n\ge 3$): 
\begin{gather}\label{HL2+} 
\|H\ke\|^2_{\L(\supp(\nabla\phi\ke))}
\leq C_1\eps^{n}\eta^2\e=C_1\eps^{n+4}.
\end{gather}
Similarly, using  Lemma~\ref{lemma:Hest} (note that $\supp(\phi_{1,\eps}-1)\subset\R^n\setminus\B(\ell_1\eps,z_1)$, whence the estimates \eqref{H:estimates} are valid due to Remark~\ref{rem:Hest}), we  get
\begin{gather}\label{nablaHL2}
\|\nabla H_{1,\eps}\|^2_{\L(\supp(\phi_{1,\eps}-1))}
\leq 
C\eps^{n+2}.
\end{gather}
Combining \eqref{inter:1}--\eqref{nablaHL2}, we obtain
\begin{gather}\label{U1}
U_{1,\eps}
\leq 
C\eps^{n+2}.
\end{gather}
Furthermore, using \eqref{nablaHL2} and \eqref{U1}, we arrive at the estimate for $U_{2,\eps}$:
\begin{gather}\label{U2}
|U_{2,\eps}|\leq 
2\|\nabla H_{1,\eps}\|_{\L(\supp(\phi_{1,\eps}-1))}
\sqrt{U_{1,\eps}} \leq 
C\eps^{n+2}.      
\end{gather}
Combining \eqref{U1} and \eqref{U2}, we infer
\begin{gather}\label{ae:asy1}
\a\e[v\e,v\e]=
{1\over 4}|B_{1,\eps}|^{-1}\left(\capty( {D_{1,\eps}})+\mathcal{O}(\eps^{n+2}) \right)=
\gamma_{1,\eps}+\mathcal{O}(\eps^2)=
\Lambda'-\eps^{1/2}+\mathcal{O}(\eps^2).
\end{gather}
Similarly, we derive the asymptotics
\begin{gather}\label{ae:asy2}
\|v\e\|^2_{\HS\e}=
1+\mathcal{O}(\eps).
\end{gather}
The required estimate \eqref{ve:RelQ} follows  from \eqref{ae:asy1}--\eqref{ae:asy2}.

\section{Waveguide with prescribed eigenvalues\label{sec:5}} 

The spectrum of the Dirichlet Laplacian on the straight tubular domain  
coincides with $[\Lambda',\infty)$, where $\Lambda'>0$ is the smallest eigenvalue of the Dirichlet Laplacian on the tube cross-section. Let us perturb this tube by narrowing   
its bounded part and then by inserting $m$ resonators within this narrowed part.
Such a perturbation does not change the essential spectrum, but may produce 
discrete eigenvalues below $\Lambda'$. Our goal is to show that these eigenvalues can be made coinciding with prescribed numbers via a suitable choice of the parameters $d\ke$.
Below we formulate the problem and the   result precisely.

Let $\Omega'$ and $\Omega''$  be bounded Lipschitz domains in $\R^{n-1}$ ($n\ge 2$) such that $\overline{\Omega''}\subset\Omega'$.
Let $L>0$. We introduce the following domains in $\R^n$:
$$
\wt\Omega\ceq \Omega''\times (-L,L),\quad 
\Omega^+\ceq \Omega'\times (L,\infty),\quad 
\Omega^-\ceq \Omega'\times(-\infty,-L),\quad
S^\pm=\Omega''\times\{\pm L\}.
$$
Further, let $S\ke\subset \wt\Omega$, $k\in\M=\{1,\dots,m\}$
be the sets we introduced in Section~\ref{sec:2}, namely
\begin{gather*}
S\ke=(\partial(\underbrace{\eps B_k + z_k}_{B\ke}))\setminus (\underbrace{d\ke D_k+z_k}_{D\ke}).
\end{gather*}
Here $\eps>0$,
$(z_k)_{k\in\M}$ are pairwise distinct points in $\wt\Omega$, 
$(B_k)_{k\in\M}$ and $(D_k)_{k\in\M}$ 
are sets in $\R^n$ satisfying \eqref{DB1}--\eqref{DB2}, finally, the numbers $(d\ke)_{k\in\M}$   are specified as follows:
\begin{gather} 
 d\ke=
\begin{cases}
 d_k\eps^{n\over n-2},&n\ge 3,
\\
\exp\left(-{1\over   d_k\eps^2}\right),&n=2,
\end{cases}\label{dke}
\end{gather}
where $d_k$ are positive constants. The parameter $\eps$ is supposed to be sufficiently small 
in order to have
$\overline{B\ke}\subset \wt\Omega $ and $\overline{B\ke}\cap \overline{B_{l,\eps}}=\emptyset$ if $k\not=l$.
We set
$$\wt\Omega\e \ceq \wt\Omega\setminus \left(\cupl_{k\in\M} S\ke\right).$$
Finally, we connect $\wt\Omega\e$ (via $S^\pm$) with $\Omega^-$ and $\Omega^+$, and
arrive at the domain  (see Figure~\ref{fig-waveguide})
$$
\Omega\e\ceq   \Omega^-\cup S^-\cup\wt\Omega\e\cup S^+\cup\Omega^+
$$
We also introduce the set 
$$
\Omega\ceq \Omega^-\cup S^-\cup\wt\Omega\cup S^+\cup\Omega^+.
$$

As before,   $\A\e$ stands for the operator acting in $\L(\Omega\e)$ and
being associated with the sesquilinear form \eqref{ae}, i.e. $\A\e$ is the Laplace operator on $\Omega\e$ subject to the Neumann boundary conditions on $\cup_{k\in\M}S\ke$ and the Dirichlet  boundary  conditions on  $\partial\Omega$.
Using standard methods of perturbation theory  one can easily demonstrate that the essential spectrum of $\A\e$ coincides with the essential spectrum of 
the Dirichlet Laplacian $-\Delta^D_{\Omega'\times\R}$ on the unperturbed  waveguide $\Omega'\times\R$,  
i.e.
\begin{gather}\label{ess:ess:1}
\sigma_\ess(\A\e)=\sigma_\ess(-\Delta_{\Omega'\times\R}^D)=[\Lambda',\infty),
\text{ where } \Lambda'\ceq \lambda_1(-\Delta_{\Omega'}^D).
\end{gather}
Thus  the discrete spectrum of $\A\e$ (if any) belongs to $[0,\Lambda')$.
The theorem below asserts that, under  a suitable choice  of the constants $d_k$ in \eqref{dke},
$\sigma_{\disc}(\A\e)$ has $m$ simple eigenvalues, which coincide with prescribed numbers.
The role of the local narrowing is to guaranteed that
there are no further eigenvalues in the vicinity of $\Lambda'$.
Note that solely  the local narrowing will not change the spectrum. Indeed, 
similarly to \eqref{ess:ess:1}, we get
\begin{gather*}
\sigma_\ess(-\Delta_{\Omega}^D)=\sigma_\ess(-\Delta_{\Omega'\times\R}^D)=[\Lambda',\infty).
\end{gather*}
Moreover, $-\Delta_{\Omega}^D$ has no eigenvalues below $\Lambda'$  
(this follows easily from 
$\sigma_{\disc}(-\Delta_{\Omega'\times\R}^D)=\varnothing$, 
$\Omega\subset\Omega'\times\R$
and the min-max principle).

To formulate the result we introduce the functions $\mathcal{F}_k:\R_+\to\R_+$, $k\in\M$, via
\begin{gather*}
\mathcal{F}_k(t)=
\begin{cases}
\ds\left(t{4|B_k|\over \mathrm{cap}(D_k)}\right)^{1\over n-2} ,&n\ge 3,\\
\ds t{2|B_k| \over \pi},&n=2.
\end{cases}
\end{gather*}
In the following, in order to emphasize the dependence of $\A\e$ (and its spectrum) on $ d_1,\dots,d_m$,
we will use the notation $\A\e^{d_1,\dots,d_m}$ instead of $\A\e$.

\begin{theorem}\label{th:exact}
Let $\wt\gamma_k$, $k\in\M$ be arbitrary numbers satisfying
\begin{gather*}
0<\wt\ga_1<\wt\ga_2<\dots<\wt\ga_m<\Lambda'.
\end{gather*}
Then for any $\delta>0$ there exist $\eps=\eps(\delta)>0$ and $\wt d_k\in (\mathcal{F}_k(\wt\ga_k)-\delta,\mathcal{F}_k(\wt\ga_k)+\delta)$, $k\in\M$
such that
$$
\sigma_{\disc}(\A\e^{\wt d_1,\dots,\wt d_m})=\cup_{k\in\M}\{\wt\ga_k \},
$$
and
 the eigenvalues $\wt\gamma_k$ are simple.
\end{theorem}

\begin{remark}\label{rem:almost}
The above theorem asserts that  $\wt d_k$
can be indicated ``almost explicitly'': they belong to the $\delta$-neighbourhoods of
the numbers $\mathcal{F}_k(\wt\gamma_k)$, where $\delta>0$ can be chosen arbitrary small.
Of course, the smaller $\delta$ is chosen, the smaller $\eps(\delta)$ will be.
\end{remark}

\begin{proof}[Proof of Theorem~\ref{th:exact}]
Let us fix $\delta>0$.
Let us also six sufficiently small $\eta>0$ satisfying
\begin{gather}\label{eta:smallness:1}
\forall k\in\M:\quad [\wt\gamma_k-2\eta,\wt\gamma_k+2\eta]\subset (0,\Lambda'),\\[2mm]
\label{eta:smallness:2}
\forall k,l\in\M,\ k\not=l:\quad
[\wt\gamma_k-2\eta,\wt\gamma_k+2\eta]\cap[\wt\gamma_l-2\eta,\wt\gamma_l+2\eta]=\emptyset,\\[2mm]
\label{ddist}
\max_{k\in\M}(\mathcal{F}_k(\wt\gamma_k+\eta)-\mathcal{F}_k(\wt\gamma_k-\eta))<\delta
\end{gather}
(the last property can be achieved since $\mathcal{F}_k$, $k\in\M$ are  continuous functions).
We denote
$$
 d_k^\pm\ceq \mathcal{F}_k(\wt\gamma_k\pm\eta).    
$$
Due to a strict 
monotonicity of $\mathcal{F}_k$, 
we have 
$0< d_k^-<d_k^+$. Finally, we introduce the set 
$$
\mathcal{D}\ceq \prod\limits_{k=1}^m [d^-_k, d^+_k]\subset\R^m.
$$

By virtue of Theorems~\ref{th1}--\ref{th2} one has for each $(d_1,\dots,d_m)\in\mathcal{D}$:
\begin{gather}\label{recall}
\begin{array}{lr}
\sigma(\A\e^{d_1,\dots,d_m})\to \sigma(\A^{d_1,\dots,d_m}),& \\[2mm]
\sigma_\disc(\A\e^{d_1,\dots,d_m})\nearrow \sigma_\disc(\A^{d_1,\dots,d_m})
&\text{with the multiplicity being preserved}
\\ 
&\text{(in the sense as in Theorem~\ref{th2}).}
\end{array}
\end{gather}
Here $\A^{d_1,\dots,d_m}$ is an operator in $\L(\Omega)\oplus\C^m$
given by
$$\A^{d_1,\dots,d_m}=(-\Delta_{\Omega}^D)\oplus\mathrm{diag}\{\gamma_1^{d_1},\dots,\gamma_m^{d_m}\},$$
with
$\gamma_k^{d_k}=\lim_{\eps\to 0}\capty(D\ke)|B\ke|^{-1}$. It follows from
\eqref{dke} and \eqref{cap:rescaling1}--\eqref{cap:rescaling2} that
$$
\gamma_k^{d_k}=\mathcal{F}^*_k(d_k),
$$
where  $\mathcal{F}^*_k:\R_+\to\R_+$ is a strictly increasing function defined by
$$
\mathcal{F}^*_k(s)=
\begin{cases}
\ds s^{n-2}{\capty(D_k)\over 4|B_k|},&n\ge 3,\\
\ds s {\pi\over 2|B_k|},&n=2.
\end{cases}
$$
Observe that 
$\mathcal{F}_k(\mathcal{F}^*_k(s))=s$ and $\mathcal{F}^*_k(\mathcal{F}_k(t))=t$;
we have
\begin{gather}\label{gagaga}
\begin{array}{c}
\ga_k^{d_k^-}=\wt\gamma_k-\eta< \wt\gamma_k< \wt\ga_k+\eta=\gamma_k^{d_k^+},\quad
\gamma_k^{d_k}\in [\wt\gamma_k-\eta,\wt\gamma_k+\eta]\text{ as }d_k\in [d_k^-,d_k^+].
\end{array}
\end{gather}

Let us fix some $\wh\Lambda\in(\gamma_m+2\eta,\Lambda')$.

\begin{lemma}\label{lm:spec:struct}
There exists $\eps'>0$ such that
for any $\eps\in(0,\eps']$ and $(d_1,\dots,d_m)\in\mathcal{D}$ the spectrum of 
$\A\e^{d_1,\dots,d_m}$ has the following structure within $ [0,\wh\Lambda]$:
\begin{gather}
\label{spec:d}    
\sigma(\A\e^{d_1,\dots,d_m})\cap [0,\wh\Lambda]=\cupl_{k\in\M}\{\gamma_{k,\eps}^{d_1,\dots,d_m}\},
\end{gather}
where the numbers $\gamma_{k,\eps}^{d_1,\dots,d_m}$ are simple eigenvalues satisfying 
\begin{gather}\label{3/2}
\wt\gamma_{k}-{3\eta\over 2}\leq \gamma_{k,\eps}^{d_1,\dots,d_m}\leq \wt\gamma_{k}+{3\eta\over 2}.
\end{gather}
Moreover, for each $k\in\M$ one has
\begin{gather}
\label{Hempel:ineq}
\gamma\ke^{d_1^+,d_2^+,\dots,d_{k-1}^+,d_k^-,d_{k+1}^+,\dots,d_m^+}
<\wt\gamma_k<
\gamma\ke^{d_1^-,d_2^-,\dots,d_{k-1}^-,d_k^+,d_{k+1}^-,\dots,d_m^-}.
\end{gather}
\end{lemma}

\begin{proof}
Due to \eqref{eta:smallness:1}, \eqref{eta:smallness:2}, \eqref{gagaga},
the intervals $[\gamma_k^{d_k^-}-{\eta\over 2}, \gamma_{k}^{d_k^-}+{\eta\over 2}]$, $k\in\M$ are pairwise disjoint and belong to $(0,\wh\Lambda)$. 
Hence, by virtue of
\eqref{ess:ess:1} and 
\eqref{recall}, for sufficiently small $\eps$ the spectrum of  $\A^{d_1^-,\dots,d_m^-}$
withing the interval $[0,\wh\Lambda]$ consists of $m$ simple eigenvalues
so that in each interval $[\gamma_k^{d_k^-}-{\eta\over 2}, \gamma_{k}^{d_k^-}+{\eta\over 2}]$ one has precisely one eigenvalue:
there exist $\eps_->0$ such that   
\begin{gather}\label{spec-}
\sigma(\A\e^{d_1^-,\dots,d_m^-})\cap[0,\wh\Lambda]=\cupl_{k\in\M}\{\gamma_{k,\eps}^-\}
\text{ for } \eps\in (0,\eps_-],
\end{gather}
where the numbers $\gamma_{k,\eps}^-$  are simple eigenvalues satisfying
\begin{gather}
\gamma_{k,\eps}^-\in 
[\gamma_k^{d_k^-}-{\eta\over 2}, \gamma_{k}^{d_k^-}+{\eta\over 2}]=
[\wt\gamma_k-{3\eta\over 2}, \wt\gamma_{k}-{\eta\over 2}].
\end{gather}
Similarly, there exist $\eps_+>0$ such that   
\begin{gather}\label{spec+}
\sigma(\A\e^{d_1^+,\dots,d_m^+})\cap [0,\wh\Lambda]=\cupl_{k\in\M}\{\gamma_{k,\eps}^+\}
\text{ for } \eps\in (0,\eps_+],
\end{gather}
where the numbers $\gamma_{k,\eps}^+$  are simple eigenvalues satisfying
\begin{gather}
\gamma_{k,\eps}^+\in [ \gamma_k^{d_k^+}-{\eta\over 2}, \gamma_{k}^{d_k^+}+{\eta\over 2}]=
[\wt\gamma_k+{\eta\over 2}, \wt\gamma_{k}+{3\eta\over 2}].
\end{gather}
Let $\a\e^{d_1,\dots,d_m}$ be the sesquilinear form associated 
with  $\A\e^{d_1,\dots,d_m}$.
It is easy to see that for any
$(d^1_1,\dots,d^1_m)\in\mathcal{D}$,
$(d^2_1,\dots,d^2_m)\in\mathcal{D}$ satisfying $d_k^1\leq d_k^2$, one has
\begin{gather}\label{a:monoton}
\begin{array}{c}
\dom(\a\e^{d^1_1,\dots,d^1_m})\supset \dom(\a\e^{d_1^2,\dots,d_m^2}),\\[2mm]
\forall u\in\dom(\a\e^{d^2_1,\dots,d^2_m}): \ \a\e^{d_1^1,\dots,d_m^1}[u,u]=
\a\e^{d_1^2,\dots,d_m^2}[u,u].
\end{array}
\end{gather}
Then, using the min-max principle, we conclude  from \eqref{spec-}--\eqref{a:monoton} that  
\eqref{spec:d}--\eqref{3/2} hold for $\eps\le \min\{\eps_-,\eps_+\}$; moreover,
by Theorem~\ref{th1},
$
\gamma_{k,\eps}^{d_1,\dots,d_m}\to \gamma_{k}^{d_k}\text{ as }\eps\to 0$.
In particular, one has
\begin{gather}\label{Hempel:lim}
\lim_{\eps\to 0}\gamma\ke^{d_1^+,d_2^+,\dots,d_{k-1}^+,d_k^-,d_{k+1}^+,\dots,d_m^+}=\gamma_k^{d_k^-},
\quad
\lim_{\eps\to 0}\gamma\ke^{d_1^-,d_2^-,\dots,d_{k-1}^-,d_k^+,d_{k+1}^-,\dots,d_m^-}=\gamma_k^{d_k^+}.
\end{gather}
Due to \eqref{gagaga} and \eqref{Hempel:lim} there exists $\eps'\le \min\{\eps_-,\eps_+\}$
such that 
\eqref{Hempel:ineq} holds
for $\eps\in (0,\eps']$.
The lemma is proven.
\end{proof}

\begin{lemma}\label{lm:no:spec}
There exists $\wh\eps\in (0,\eps']$ such that $(d_1,\dots,d_m)\in\mathcal{D}$ and 
for any $\eps\in(0,\wh\eps]$    one has 
\begin{gather}\label{nospectrum}
\sigma(\A\e^{d_1,\dots,d_m})\cap (\wh\Lambda,\Lambda')=\emptyset.
\end{gather}
\end{lemma}

\begin{proof}
For the proof we use bracketing technique.
We represent $\L(\Omega\e)$ 
as a direct sum
$$\L(\Omega\e)=\L(\Omega^-)\oplus \L(\wt\Omega\e)\oplus\L(\Omega^+).$$
With respect to this space decomposition we introduce the operator
$$\wh\A\e^{d_1,\dots,d_m}=\A^-\oplus\wt\A\e^{d_1,\dots,d_m}\oplus\A^+,$$ 
where
\begin{itemize}

\item the operator $\wt\A\e^{d_1,\dots,d_m}$ is associated with the sesquilinear form  
$\wt\a\e  $ given by\\
$\ds\wt\a\e [u,v]=\int_{\wt\Omega\e}\nabla u\cdot\overline{\nabla v}\d x,\ \dom(\wt\a\e )=\{u\in \H^1(\wt\Omega\e):\ u=0\text{ on }\partial\Omega''\times(-L,L)\}.$\\
In other words,   $\wt\A\e^{d_1,\dots,d_m}$ is the Laplacian on $\wt\Omega\e$
subject to the Dirichlet boundary conditions on $\partial\Omega''\times (-L,L)$ and the Neumann boundary conditions on $S^\pm$.

\item the operator $\A^-$ (respectively, $\A^+$) in the Laplace operator on $\Omega^-$ (respectively, on $\Omega^+$) subject to the
Neumann boundary conditions on 
$S^-$  (respectively, on $S^+$), and the Dirichlet boundary conditions on 
$(\partial\Omega^-)\setminus S^-$  (respectively, on $(\partial\Omega^+)\setminus S^+$).
\end{itemize}
We denote by $\wt\A$ the Laplacian on $\wt\Omega$
subject to the
Neumann boundary conditions on 
$S^\pm$, and the Dirichlet boundary conditions on 
$\partial\Omega''\times (-L,L)$.
It is easy to see that
\begin{gather}\label{inf''}
\inf\sigma(\wt\A)=\Lambda''\ceq\lambda_1(-\Delta_{\Omega''}^D).
\end{gather}
Also, since  $\overline{\Omega''}\subset\Omega'$, we get  (see, e.g, \cite[Subsection~1.3.2]{He06})
\begin{gather}\label{LaLa}
\Lambda'<\Lambda''.
\end{gather}
By virtue of Theorems~\ref{th1}--\ref{th2} (taking into account Remark~\ref{rem:Neumann}) we have
\begin{gather}\label{recall:wt}
\begin{array}{lr}
\sigma(\wt\A\e^{d_1,\dots,d_m})\to \sigma(\wt\A^{d_1,\dots,d_m}),&
\\[2mm]
\sigma_\disc(\wt\A\e^{d_1,\dots,d_m})\nearrow \sigma_\disc(\wt\A^{d_1,\dots,d_m}) 
&\text{with the multiplicity being preserved}  
\\
&\text{(in the sense as in Theorem~\ref{th2})},
\end{array}
\end{gather}
 where $\wt\A^{d_1,\dots,d_m}=\wt\A\oplus\mathrm{diag}\{\gamma_1^{d_1},\dots,\gamma_m^{d_m}\}.$
Combining \eqref{inf''}--\eqref{recall:wt}  and using the same arguments as in the proof of \eqref{spec:d},
we conclude there exists  $\wt\eps>0$  such that
$ \sigma(\wt\A\e^{d_1,\dots,d_m})\cap[0,\Lambda')$
consists of precisely $m$ simple eigenvalues
provided $\eps\in (0,\wt\eps]$ and $(d_1,\dots,d_m)\in\mathcal{D}$.
Hence, since $\sigma(\A^\pm)=[\Lambda',\infty)$, we infer
\begin{gather}\label{nospectrum:wh}
\sigma(\wh\A\e^{d_1,\dots,d_m})\cap[0,\Lambda')\text{ consists of }m\text{ simple eigenvalues}
\end{gather}
provided $\eps\in(0,\wt\eps]$ and $(d_1,\dots,d_m)\in\mathcal{D}$.
Finally, we observe that $\wh\A\e^{d_1,\dots,d_m}\leq \A\e^{d_1,\dots,d_m}$ in the form sense. Using this fact and the min-max principle, we conclude from \eqref{nospectrum:wh} that
\begin{multline*}
\sigma(\A\e^{d_1,\dots,d_m})\cap[0,\Lambda')\text{ consists of \emph{at most} }m\text{ eigenvalues} \\\text{(with multiplicities taken into account)}
\end{multline*}
for $\eps\in(0,\wt\eps]$, $(d_1,\dots,d_m)\subset\mathcal{D}$.
On the other hand (see \eqref{spec:d}) $\sigma(\A\e^{d_1,\dots,d_m})\cap[0,\wt\Lambda]$ consists of 
precisely $m$ simple eigenvalues for $\eps\in(0,\eps']$.
Hence \eqref{nospectrum} holds with $\wh\eps\ceq\min\{\eps',\wt\eps\}$. 
\end{proof}

We need the following multi-dimensional version
of the intermediate value theorem. 

\begin{lemma}[{\cite[Lemma~3.5]{HKP97}}]\label{lemma-hempel}
Let $\mathcal{D}=\Pi_{k=1}^n[a_k, b_k]$ with $a_k < b_k$, $k=1,\dots,n$, assume that 
$f:\mathcal{D}\to\R^n$ is continuous and each coordinate function $f_k$ of $f$ is
monotonically increasing in each of its arguments. If
$F_k^-<F_k^+$, $k=1,\dots,n$, where
\begin{gather*}
F_k^-\ceq f_k(b_1,b_2,\dots,b_{k-1},a_k,b_{k+1},\dots,b_n),\\
F_k^+\ceq f_k(a_1,a_2,\dots,a_{k-1},b_k,a_{k+1},\dots,b_n),
\end{gather*}
then for any $F\in\Pi_{k=1}^n[F_k^-,F_k^+]$
there exists  $x\in\mathcal{D}$ such that $f(x)=F$.
\end{lemma}

Now, we are ready to finish the proof of Theorem~\ref{th:exact}.
Let us \textit{fix} $\eps\in (0,\wh\eps]$; with this $\eps$ the properties \eqref{spec:d}--\eqref{Hempel:ineq}, \eqref{nospectrum} hold true.
We introduce  $f=(f_1,\dots,f_m):\mathcal{D}\to\R^m$ 
by
$$
f_k(d_1,\dots,d_m)= \gamma\ke^{d_1,\dots,d_m}.
$$
By the min-max principle, each $f_k$   is
monotonically increasing in each of its arguments. 
Moreover, each $f_k$ depends continuously on $(d_1,\dots,d_m)\in\mathcal{D}$
(for $n=2$ the proof can be found in \cite[Theorem~A.1]{BK19}, in the general case the proof is similar).
Finally, due to \eqref{Hempel:ineq}, one has for each $k\in\M$:
\begin{gather*}
f_k( d_1^+,\dots, d_{k-1}^+,  d_k^-,  d_{k+1}^+,\dots,\dots, d_m^+)<
\wt\ga_k<
f_k( d_1^-,\dots, d_{k-1}^-, d_k^+,  d_{k+1}^-,\dots,\dots,d_m^-)
\end{gather*}
Then, by Lemma~\ref{lemma-hempel}, there exists $(\wt d_1,\dots,\wt d_m)\in\mathcal{D}$ such
that $f_k(\wt d_1,\dots,\wt d_m)=\wt\gamma_k$, $\forall k\in\M$. Thus, taking into account \eqref{ess:ess:1} and \eqref{nospectrum}, we arrive at
$$\sigma_{\disc}( \A\e^{\wt d_1,\dots,\wt d_m})=\cup_{k\in\M}\{\wt\ga_k\}.$$
Evidently, all eigenvalues $\wt\ga_k$ are simple.
Finally, using \eqref{ddist} and taking into account that $\mathcal{F}_k(\wt\ga_k)\in   [\mathcal{F}_k(\wt\gamma_k-\eta),\mathcal{F}_k(\wt\gamma_k+\eta)]=[d_k^-,d_k^+]$, 
we get 
\begin{gather*}
|\wt d_k - \mathcal{F}_k(\wt\ga_k)|\leq d_k^+-d_k^-=\mathcal{F}_k(\wt\ga_k+\eta)-\mathcal{F}_k(\wt\ga_k-\eta)<\delta
\end{gather*}
Theorem~\ref{th:exact} is proven.
\end{proof}

\appendix

\section{\label{appendix:A}Convergence of operators in varying Hilbert spaces}

In this {appendix} we present an abstract scheme for studying the convergence of operators
in varying Hilbert spaces proposed by Post in \cite{P06}, and further elaborated in the monograph \cite{P12}. 
Initially, this scheme was applied to study convergence of the Laplace-Beltrami operator on ``fat'' graphs; 
later it also has shown to be effective to investigate resolvent and spectral convergence 
in domains with holes \cite{KPl21,KP18,AP21}.

{
As in Section~\ref{sec:3},
let $(\HS\e)_{\eps>0}$ be a family of Hilbert spaces, 
$(\A\e)_{\eps>0}$ be a family of non-negative, self-adjoint, unbounded operators in $\HS\e$, $(\a\e)_{\eps>0}$ be a family of associated 
sesquilinear forms. Similarly, let $\HS$ be a Hilbert space,  $\A$ be a non-negative, self-adjoint, unbounded operator in $\HS$, $\a$ be the associated 
sesquilinear form.  
We define the Hilbert spaces $\HS^1\e$, $\HS^1$ and $\HS^2$  
by \eqref{scale:1}--\eqref{scale:2}.}

The definition below generalizes the standard notion of norm resolvent convergence 
to the setting of varying Hilbert spaces.

\begin{definition}[{\cite{P12}}]\label{def:gnrc}
	We say that $\A\e$ converges to $\A$ in the generalized norm-resolvent sense 
	($\A\e\overset{g.n.r.c}\to \A$) as $\eps\to 0$ if there is a sequence of   $(\delta\e)_{\eps>0}$ with $\delta\e\to 0$
	and linear bounded operators 
	$\J\e \colon \HS\to  \HS\e$, ${\wt\J\e }\colon {\HS\e}\to \HS,$
	satisfying the conditions
	\begin{align}
		\label{C1}   
		|(u,\J\e f)_{\HS\e} - (\wt\J\e u,f)_{\HS}| \leq \delta\e\|f\|_{\HS}\|u\|_{\HS\e},&\quad  \forall f\in\HS,\, u\in\HS\e,
		\\
		\label{C2}
		\|f-\wt\J\e \J\e f \|_{\HS} \leq \delta\e\|f\|_{\HS^1},&\quad  \forall f\in\HS^1,
		\\  \label{C3}
		\|u-\J\e\wt \J\e u \|_{\HS\e}\leq \delta\e\|u\|_{\HS\e^1},&\quad   \forall u\in \HS^1\e, 
		\\
		\label{C4} 
		\|\J\e f \|_{\HS\e} \leq 2\|f\|_{\HS},&\quad \forall f\in \HS,
		\\
		\label{C5}
		\|(\A\e+\Id)^{-1}\J\e f - \J\e (\A+\Id)^{-1} f \|_{\HS\e} \leq \delta\e\|f\|_{\HS},&\quad\forall f\in \HS.
	\end{align}
	
\end{definition}

\begin{remark}
	For the sake of clarity, we formulate Definition~\ref{def:gnrc} in a slightly different  form than the one in \cite{P12}. For example,  \eqref{C1} is written in \cite{P12} in the following equivalent form:
	$$\|\J\e-(\wt\J\e)^*\|_{\HS\to\HS\e}\leq \delta\e.$$
	In fact, the above definition is a mix of Definitions~4.1.1, 4.2.1, 4.2.3, 4.2.6 from \cite{P12}.
\end{remark}

Classical perturbation theory yields that the norm resolvent convergence of  self-adjoint operators in a \textit{fixed} Hilbert space implies the convergence of their spectra. 
The theorem below extends this result to the setting of varying Hilbert spaces.

\begin{theorem}[{\cite[Theorems~4.3.3 \& 4.3.5]{P12}}]
	\label{thA1}
	Let $\A\e\overset{g.n.r.c}\to \A$ as $\eps\to 0$.
	Then
	$$
	\sigma(\A\e)\to \sigma(\A),\quad
	\sigma_\disc(\A\e)\nearrow\sigma_\disc(\A)
	\quad\text{as}\quad\eps\to 0.$$
	The multiplicity is preserved {in the following sense}: if $\lambda\in\sigma_\disc(\A)$ is of multiplicity $\mu$ and
	$[\lambda-L,\lambda+L]\cap\sigma(\A)=\{\lambda\}$ with $L>0$, then 
	for sufficiently small $\eps$  the spectrum of  
	$\A\e$ in $[\lambda-L,\lambda+L]$ is purely discrete 
	and 
	the total multiplicity of the eigenvalues of $\A\e$ contained in $[\lambda-L,\lambda+L]$ equals $\mu$.
	
	If, in addition, $\mu=1$ (i.e.  the eigenvalue $\lambda$ is simple), and  $\psi$ is the corresponding normalized in $\HS$ eigenfunction,  
	then there exists a sequence of   normalized in $\HS\e$ eigenfunctions $\psi\e$ of $\A\e$  such that $$\|\psi\e - \J\e \psi \|_{\HS\e}\to 0\text{ as }\eps\to 0.$$
\end{theorem}

For two bounded normal operators $\Res\e$ and $\Res$   in a \textit{fixed} Hilbert space $\HS$ 
one has the following estimate \cite[Lemma~A.1]{HN99}:
$$d_{H}(\sigma(\Res\e),\sigma(\Res))\leq \|\Res\e-\Res\|_{\HS\to\HS}.$$
Applying it for the resolvents $\Res\e\ceq (\A\e+\Id)^{-1}$ and $\Res \ceq (\A+\Id)^{-1}$ of non-negative self-adjoint operators $\A\e$, $\A$ acting in $\HS$ and taking into account that  by spectral mapping theorem $\wt{d}_{H}(\sigma(\A\e),\sigma(\A))=d_{H}(\sigma(\Res\e),\sigma(\Res))$, we get
$$\wt d_{H}(\sigma(\A\e),\sigma(\A))\leq \|(\A\e+\Id)^{-1}-(\A+\Id)^{-1}\|_{\HS\to\HS}.$$
The theorem below is an analogue of this result to the setting of varying spaces.

\begin{theorem}[{\cite[Theorem~3.4]{KP21}}]\label{thA2}
	Let  
	$\J\e \colon\HS\to {\HS\e}$, $\wt\J\e \colon\HS\e\to {\HS} $ be  linear bounded operators satisfying 
	\begin{align*}
		\|(\A\e+\Id)^{-1}\J\e - \J\e (\A+\Id)^{-1} \|_{\HS\to \HS\e}\leq \rho\e,
		\\
		\|\wt\J\e (\A\e+\Id)^{-1} - (\A+\Id)^{-1}\wt\J\e  \|_{\HS\e\to  \HS}\leq \wt\rho\e ,
	\end{align*}
	and, moreover,
	\begin{align}
		\label{thA2:3}
		\|f\|^2_{ \HS}&\leq \mu\e \|\J\e  f\|^2_{\HS\e}+\nu\e \,  \a[f,f],\quad \forall f\in \dom( \a),\\
		\label{thA2:4}
		\|u\|^2_{\HS\e}&\leq \wt\mu\e \|\wt\J\e  u\|^2_{\HS}+\wt\nu\e \,  
		\a\e[u,u],\quad \forall u\in \dom( \a\e)
	\end{align}
	for some positive constants $\rho\e ,\,\mu\e ,\,\nu\e ,\, \wt\rho\e ,\,\wt\mu\e ,\,\wt\nu\e  $.
	Then   one has
	\begin{gather}\label{thA2:est}
		\widetilde{d}_H\left(\sigma(\A\e),\,\sigma(\A)\right)\leq 
		\max
		\left\{
		{\nu\e\over 2}+\sqrt{{\nu\e^2\over 4}+\rho\e^2\mu\e};\,
		{\wt\nu\e\over 2}+\sqrt{{\wt\nu\e^2\over 4}+\wt\rho\e^2\wt\mu\e}
		\right\}.
	\end{gather}
\end{theorem} 

\begin{remark}
	\label{rem:tau}
	In fact, in \cite{KP21} the obtained estimate reads
	\begin{gather}\label{rem:tau:est}
		\widetilde{d}_H\left(\sigma(\A\e),\,\sigma(\A)\right) 
		\leq 
		\max\left\{
		\rho\e\sqrt{\mu\e\over \tau};\,{\nu\e\over 1-\tau };\,
		\wt\rho\e\sqrt{\wt\mu\e\over \wt\tau};\,{\wt\nu\e\over 1-\wt\tau }
		\right\},\
		\forall\tau,\wt\tau\in (0,1).
	\end{gather}
	Minimizing
	the right-hand-side of \eqref{rem:tau:est} over $\tau,\wt\tau\in (0,1)$, 
	one gets the estimate \eqref{thA2:est}.
\end{remark}

\begin{remark}
	We established slightly weaker version of {Theorem~\ref{thA2}}   in \cite{CK19}.
\end{remark}

\begin{proposition}[{\cite[Proposition~3.8]{KP21}}]
	\label{prop:2/3}
	Let the conditions \eqref{C1} and \eqref{C2} (respectively, \eqref{C1} and \eqref{C3}) 
	{be} fulfilled with $\delta\e<2/3$. Then 
	the estimate \eqref{thA2:3} (respectively, estimate \eqref{thA2:4}) is valid with 
	$$
	\mu\e=1+{4\delta\e\over 2-3\delta\e},\
	\nu\e={\delta\e\over 2-3\delta\e}\qquad
	\text{(respectively, }\wt\mu\e=1+{4\delta\e\over 2-3\delta\e},\
	\wt\nu\e={\delta\e\over 2-3\delta\e}\text{)}. 
	$$
\end{proposition}

The last theorem gives a useful tool to establish the crucial condition \eqref{C5} in the definition of the generalized norm resolvent convergence. It is well-known (see, e.g., \cite[Theorem~VI.3.6]{Ka66} or \cite[Theorem~VIII.25]{RS72}) that convergence of sesquilinear forms with \emph{common domain} implies norm resolvent convergence of the associated operators. The theorem below is as a generalization of this fact  to the setting of varying spaces.

\begin{theorem}[{\cite[Theorem~A.5]{P06}}]
	\label{thA3}
	Let $\J\e \colon \HS\to  \HS\e$, ${\wt\J\e }\colon {\HS\e}\to \HS$ be linear bounded operators satisfying \eqref{C1}. Furthermore, let 
	$\J\e^1 \colon \HS^1\to  \HS\e^1$, ${\wt\J\e^1 }\colon {\HS^1\e}\to \HS^1$
	be linear operators satisfying
	\begin{gather}
		\label{thA3:1}
		\|\J\e^1 f-\J\e f\|_{\HS\e}\leq \delta\e\|f\|_{\HS^1 },
		\quad \forall f\in \HS^1,
		\\
		\label{thA3:2} 
		\|\wt\J\e^1 u - \wt\J\e u \|_{\HS}\leq  \delta\e\|u\|_{ \HS^1\e},
		\quad \forall u\in\HS\e^1,
		\\\label{thA3:3}
		|\a\e[u,\J^1\e f]-\a[\wt\J^{1}\e u,f]  | \leq 
		\delta\e\|f\|_{\HS^2 }\|u\|_{\HS^1\e},\quad \forall f\in \HS^2 ,\ u\in \HS^1\e .
	\end{gather}
	Then 
	\begin{align}\label{thA3:result}
		\|(\A\e+\Id)^{-1}\J\e -\J\e (\A+\Id)^{-1}\|_{\HS\to\HS\e}\leq 4\delta\e.
	\end{align}
\end{theorem} 

\begin{remark}
	Tracing the proof of \cite[Theorem~A.5]{P06} one observes that
	the estimate \eqref{thA3:result}
	remains valid if \eqref{thA3:3} is substituted 
	by the weaker condition
	\begin{gather*}
		\left|\a\e[u,\J^1\e f]-\a[\wt\J^{1}\e u,f] \right|\leq 
		\delta\e\|f\|_{\HS^2 }\|u\|_{\HS^2\e},\quad 
		\forall f\in \HS^2 ,\ u\in \HS^2\e, 
	\end{gather*} 
	where $\HS^2\e\ceq\dom(\A\e)$, $\|u\|_{\HS^2\e}\ceq\|{(\A\e+\Id) f}\|_{\HS\e}$.
	Nevertheless, in most of the  applications one is able to establish stronger condition~\eqref{thA3:3}.
\end{remark} 

{
\section{\label{appendix:B}Proof of Theorem~\ref{thAAA}}

Using the above abstract results one can easily prove Theorem~\ref{thAAA}.
By virtue of Theorem~\ref{thA3}, assumptions \eqref{AAA:1}, \eqref{AAA:5}, \eqref{AAA:6},  \eqref{AAA:7} imply   
\begin{gather}\label{ResRes:1}
	\|(\A\e+\Id)^{-1}\J\e -\J\e (\A+\Id)^{-1}\|_{\HS\to\HS\e}\leq 4\delta\e,
\end{gather}
moreover, since $((\A\e+\Id)^{-1}\J\e -\J\e (\A+\Id)^{-1})^*=\wt\J\e(\A\e+\Id)^{-1} - (\A+\Id)^{-1} \wt\J\e$ (this equality follows from $(\J\e)^*=\wt\J\e$, see~\eqref{AAA:1}), we also get
\begin{gather}\label{ResRes:2}
	\|\wt\J\e(\A\e+\Id)^{-1} - (\A+\Id)^{-1} \wt\J\e\|_{\HS\e\to\HS}\leq 4\delta\e.
\end{gather}
It follows from \eqref{AAA:1}--\eqref{AAA:4}, \eqref{ResRes:1} that
\begin{gather}\label{gnrc}
	\A\e\overset{g.n.r.c.}\to \A\text{ as }\eps\to 0,
\end{gather}
whence, by virtue of Theorem~\ref{thA1}, we deduce
the statements (i)--(iii) of Theorem~\ref{thAAA}.
Furthermore, by virtue of Proposition~\ref{prop:2/3}, 
\eqref{AAA:1}--\eqref{AAA:3} imply the estimates 
\begin{align}
	\label{ResRes:3}
	\|f\|^2_{ \HS}&\leq 5 \|\J\e  f\|^2_{\HS\e}+2\delta\e \,  \a[f,f],\quad \forall f\in \dom( \a),\\
	\label{ResRes:4}
	\|u\|^2_{\HS\e}&\leq 5\|\wt\J\e  u\|^2_{\HS}+2\delta\e \,  
	\a\e[u,u],\quad \forall u\in \dom( \a\e)
\end{align}
Here we  make use of $\delta\e\leq 1/2$ leading
to  the inequalities
$$
\mu\e\leq 5,\quad
\nu\e\leq 2\delta\e, \quad
\wt\mu\e\leq 5,\quad
\wt\nu\e\leq 2\delta\e 
$$
for the constants standing in Proposition~\ref{prop:2/3}.
By virtue of Theorem~\ref{thA2}, the estimates 
\eqref{ResRes:1}, \eqref{ResRes:2}, \eqref{ResRes:3}, \eqref{ResRes:4} imply
the last statement of Theorem~\ref{thAAA}:
\begin{gather*}
	\widetilde{d}_H\left(\sigma(\A\e),\,\sigma(\A)\right)\leq 
	{2\delta\e\over 2}+\sqrt{{(2\delta\e)^2\over 4}+5(4\delta\e)^2}= 10\delta\e.
\end{gather*}
}

\section*{Acknowledgements}

{ The authors are grateful to the anonymous referees for many useful comments which improved the manuscript considerably. }
The work of A.K. is supported by the Czech Science Foundation (GA\v{C}R) through the project 22-18739S.

\bibliographystyle{siamplain}
\bibliography{references}

\end{document}